\documentclass[a4paper,11pt]{article}
 
\usepackage{latexsym}
\usepackage{amsmath}
\usepackage{amssymb}
\usepackage{amsthm}

\usepackage{algorithmic}
\usepackage[section]{algorithm}

\usepackage{graphicx}

\usepackage{tikz}
\usetikzlibrary{arrows.meta, positioning}

\title{
 Steepest Descent Algorithm \\
for M-convex Function Minimization\\
Using Long Step Length
}

\author{
Taihei Oki%
\footnote{
Institute for Chemical Reaction Design and Discovery (ICReDD),
Hokkaido University,
Sapporo 001-0021, Japan.
\texttt{oki@icredd.hokudai.ac.jp}
}
\and
Akiyoshi Shioura%
\footnote{
Department of Industrial Engineering and Economics,
 Institute of Science Tokyo,
Tokyo 152-8550, Japan.
  \texttt{shioura.a.aa@m.titech.ac.jp}}
}

\newcommand{\suppp}{{\rm supp}\sp{+}}
\newcommand{\suppm}{{\rm supp}\sp{-}}

\def\0{\mbox{\bf 0}}
\def\1{\mbox{\bf 1}}

\def\phi{\varphi}
\def\epsilon{\varepsilon}

\def\R{{\mathbb{R}}}

\def\Z{{\mathbb{Z}}}
\def\Oh{{\rm O}}

\def\Mnat{{M$^\natural$}}

\def\Rinf{\overline{\R}}

\def\Rpminf{\R \cup \{\pm \infty\}}

\def\P*{{\rm P}_*}

\newcommand{\dom}{{\rm dom\,}}

\newcommand{\domR}{{\rm dom_\R\,}}

\newtheorem{theorem}{Theorem}[section]
\newtheorem{lemma}[theorem]{Lemma}
\newtheorem{coro}[theorem]{Corollary}
\newtheorem{prop}[theorem]{Proposition}
\newtheorem{example}[theorem]{Example}
\newtheorem{remark}[theorem]{Remark}

\makeatletter
	
	\@addtoreset{equation}{section}
\makeatother

\usepackage{amsmath}	

\begin{document}

\maketitle

\begin{abstract}
 We consider the minimization of an M-convex function, which is 
a discrete convexity concept for functions on the integer lattice points.
 It is known that a minimizer of an M-convex function can be obtained 
by the steepest descent algorithm.
 In this paper, we propose an effective use of long step length in the steepest
descent algorithm, aiming at the reduction in the running time.
 In particular, we obtain an improved time bound by using long step length.
 We also consider the constrained M-convex function minimization and 
show that long step length can be applied to a variant of
steepest descent algorithm as well. 

\end{abstract}

\section{Introduction}

Steepest descent algorithm is one of the most fundamental algorithms for convex
function minimization in real variables, which repeatedly moves a point
by a certain step length in a steepest descent direction. 
 While choosing an appropriate step length is crucial to the efficiency of the 
algorithm, this presents a fundamental dilemma: longer step lengths can speed up the
 optimization process, but they also increase the risk of overshooting an optimal
 solution. 
Balancing these two, a number of methods to determine step lengths with provable
convergence rates have been proposed \cite{BSS2006,BV2004}.

Choosing an appropriate step length is also crucial 
in discrete optimization with integer variables.
 While  unit step length is often used in algorithms for discrete optimization
problems, due to the integrality of solutions,
long step length is also used by various algorithms, including those
for network flow problems.
 We illustrate this through the successive shortest path algorithm for the
minimum cost flow problem with integral capacities 
(see, e.g., \cite[Section 9.7]{AMO1993}).

 In the successive shortest path algorithm, 
a shortest path
from the source to the sink in the residual network
is repeatedly selected, 
and a flow is augmented along the path by a certain amount;
this amount can be viewed as step length.
 By augmenting a flow by a unit amount in each iteration,
we obtain a pseudo-polynomial time bound that is proportional to
the total supply value at the source.
 Instead, we may augment a flow by the capacity of the shortest path;
while this modification does not improve the worst-case running time, 
it makes the algorithm faster in practice.
 Moreover, we can obtain a better time bound in terms of arc cost
if we augment a flow by using multiple shortest paths 
at the same time%
\footnote{
This variant of the successive shortest path algorithm
is often referred to as a primal-dual
algorithm \cite[Section 9.8]{AMO1993}.
}.
In this way, we can improve the running time of
the algorithm theoretically and practically with the aid of long step length.

\paragraph{Our Contribution}

 In this paper, 
we extend the idea of long step length to a broad class
of discrete optimization problems called 
\textit{M-convex function minimization},
and propose an effective use of long step length in 
the steepest descent algorithm.
 The successive shortest path algorithm for the minimum cost flow problem
can be seen as a special case of the steepest descent algorithm
for M-convex function minimization (cf.~Example~\ref{ex:mcf}).
 By using long step lengths,
we obtain an improved time bound
of the steepest descent algorithm,
showing that if an initial point is sufficiently
``close'' to optimality, then an optimal point 
can be computed quickly.

 Minimization of an M-convex function is one of the most fundamental problem in 
discrete convex analysis \cite{Murota96,Murota98,Murota03book}
and includes various discrete optimization problems as special cases.
 Indeed, examples of M-convex function minimization are
the minimum weight base problem of matroids and polymatroids 
(see, e.g., \cite{Fujishige05,Schrijver03}),
separable-convex resource allocation problem under a polymatroid constraint \cite{Hochbaum94,KSI13},
the minimum cost flow problem with multiple sources and demands~\cite{AMO1993},
maximization of gross-substitute valuation \cite{PaesLeme17}, and so on
(see also Examples 2.1--2.3 in Section \ref{sec:prelim}).

  M-convex function minimization can be solved in pseudo-polynomial time
by the steepest descent algorithm
\cite{Sh1998}.
 Polynomial-time solvability of this problem is shown in \cite{Sh1998},
and various polynomial-time algorithms based on scaling approach are proposed 
in \cite{MoriguchiMurotaShioura02,Shioura03,Tamura05}.

 An optimality condition of M-convex function minimization
is described by using directions of the form%
\footnote{
$\chi_i \in \Z^n$ denotes
the $i$-th unit vector for $i=1,2,\ldots, n$. 
}
 $+ \chi_i - \chi_j$:
a vector $x$ is a minimizer of an M-convex function $f$
if and only if the slopes of $f$ in all such directions 
are non-negative \cite{Murota96,Murota98}.
 Hence, an optimal point can be obtained by
iteratively moving a point along a steepest descent direction 
(i.e., a direction with the minimum slope)
until the slope in the steepest descent direction becomes non-negative;
this is the \emph{steepest descent algorithm} for M-convex function minimization
\cite{Murota03}.
 It is shown \cite{MinamikawaShioura21,Shioura2022}
that the number of iterations in the steepest descent algorithm
is exactly equal to $(1/2)\|x^* - x_0\|_1$,
where $x_0$ is the initial point and
$x^*$ is an optimal point nearest to $x_0$ in terms of the $\ell_1$-distance.

 In the basic steepest descent algorithm,
the unit length is always chosen as the step length.
 We propose a modifined version of the steepest descent algorithm,
in which  the step length can be taken larger as far as the slope in the direction 
remains the same.
 We first show that the trajectory of the vector generated by the
modified algorithm coincides with that of 
a special implementation of the basic steepest descent algorithm.
 This modification possibly reduces the running time of the algorithm
in practice, although the theoretical worst-case time bound is the same as before.

  To obtain an improved theoretical time bound,
in Section \ref{sec:Mconv} we further
 modify the algorithm by 
selecting steepest descent directions in some specific order,
and show that, under the assumption that $f$ is integer-valued,
the number of iterations is at most
\[
 \min\{n^2 |\mu(x_0)|, (1/2)\|x^* - x_0\|_1 \},
\]
where $\mu(x_0)$ is the slope in the steepest descent direction
at the initial point.
 This bound shows that the steepest descent algorithm
using long step lengths terminates quickly if 
the slope $\mu(x_0)$ in the steepest descent direction at $x_0$
is sufficiently close to zero, or
the initial point $x_0$ is close to the set of optimal points
(or both).

We also discuss in Section \ref{sec:constM} application of long step length 
to a class of constrained M-convex function minimization problems.
 In this problem, we are given an M-convex function $f$,
an index set $R\subseteq \{1,2,\ldots, n\}$, and an integer $k$, 
and find a minimizer of $f$ under the constraint $\sum_{i \in R} x(i)=k$.
 This constrained problem is a generalization of the one for matroids
discussed by  Gabow and Tarjan \cite{GT84} and Gusfield \cite{Gusfield84}
and for polymatroids by Gottschalk et al.~\cite{GLPW18}.
 In addition, this constrained problem is used
in \cite{Shioura2022} to reformulate
a nonlinear integer programming problem arising from 
re-allocation of dock-capacity in a bike sharing system discussed by
Freund et al.~\cite{FHS2022}.
 A more general (but essentially equivalent) constrained problem 
is considered by Takazawa \cite{Takazawa2023}
(see Section \ref{sec:connect-const-Mnat}).
   The constrained M-convex function minimization can be also solved by
a variant of the steepest descent algorithm \cite{Takazawa2023}, in which the value
$\sum_{i \in R} x(i)$ is iteratively increased until it reaches $k$.
 We show that long step length can be also used in this algorithm to 
improve its time bound.

  The concept of M-convexity for functions on $\Z^n$ is extended to
ordinary (polyhedral) convex functions on $\R^n$ \cite{MS99,MLconj04}.
 We consider in Section \ref{sec:polyM} the minimization of
a polyhedral M-convex function.
 It is known (cf.~\cite{MS99}) 
that a steepest descent algorithm 
for M-convex functions on $\Z^n$ is 
applicable to polyhedral M-convex functions on $\R^n$.
 In particular, use of long step length is quite natural for
polyhedral M-convex functions.
 While the steepest descent algorithm finds a minimizer of a polyhedral M-convex function
if it terminates, it is not known so far whether the algorithm terminates
in a finite number of iterations.
 We show that  
in a variant of the steepest descent algorithm with long step length,
the slope in the steepest descent direction increases strictly after $\Oh(n^2)$ iterations.
 By using this property, we can obtain the first result 
on the finite termination of an exact algorithm for finding a global minimizer of a polyhedral M-convex functions.

\paragraph{Related Work}

Steepest descent algorithms
using long step lengths
have been proposed for another class of discrete convex functions
called L-convex functions \cite{Murota98,Murota03book}.
  L-convex functions form the conjugate class of M-convex functions
with respect to the Legendre--Fenchel transformation.
 It is known that a minimizer of an L-convex function can be obtained
by a steepest descent algorithm of a different type \cite{Murota2000,Murota03},
where a set of 0-1 vectors are used as moving directions of a point.
 For this algorithm, the long step technique is applied and similar results 
are obtained  (see \cite{Shioura2017}; see also \cite{FujishigeMurotaShioura2015}).
 In particular, it is shown that 
a variant of long-step steepest descent algorithm
that always selects a ``minimal'' steepest descent direction 
achieves a better theoretical time bound.

\paragraph{Organization of This Paper}

The concept of M-convex function and other related concepts
are explained in Section \ref{sec:prelim}. 
 Long step length is applied to the steepest descent algorithm
for M-convex function minimization in Section \ref{sec:Mconv}.
 Application to
constrained  M-convex function minimization is considered in Section \ref{sec:constM}.
 Results for polyhedral M-convex function minimization are presented in
Section \ref{sec:polyM}. 
 Omitted proofs are given in Appendix.


\section{Preliminaries}
\label{sec:prelim}

 In this section, we explain the definition of M-convex function
and other related concepts.
 
 Throughout this paper, let $n$ be a positive integer
and denote $N =\{1,2,\ldots, n\}$.
 We denote by $\Z$ and $\R$ the sets of integers and real numbers, respectively.
 Also, we denote by $\Z_+$ the set of non-negative integers,
and  $\Rinf=\R\cup\{+ \infty\}$.
  We denote $\0=(0,0,\ldots, 0)$.
 The $i$-th unit vector for $i \in N$ is denoted as $\chi_i \in \{0,1\}^n$,
i.e., $\chi_i(i)=1$ and $\chi_i(j)=0$ if $j\ne i$;
  in addition, we denote $\chi_0 = \0$.
 For a vector $x \in \R^n$, we define
$\|x\|_1 = \sum_{i \in N} |x(i)|$, $\|x\|_\infty = \max_{i \in N} |x(i)|$,
and $x(R) = \sum_{i \in R} x(i)$ for $R \subseteq N$.

 A univariate function $\psi: \Z \to \Rinf$ is said to be \emph{convex}
if it satisfies $\psi(\alpha - 1) + \psi(\alpha + 1) \ge 2 \psi(\alpha)$
for every $\alpha \in \Z$ with $\psi(\alpha) < + \infty$.
  We assume in this paper that the value of 
a function $f: \Z^n \to \Rinf$ (or $f: \R^n \to \Rinf$)
can be evaluated in constant time.

\paragraph{Base polyhedron and polymatroid}

 Let $\rho: 2^N \to \Z\cup\{+\infty\}$ 
be an integer-valued submodular function; $\rho$ is
 submodular if it  
satisfies the submodular inequality:
\[
\rho(X) + \rho(Y) \ge \rho(X \cap Y) + \rho(X \cup Y) \qquad (X, Y \in 2^N).
\]
 In this paper, we assume $\rho(\emptyset) =0$ for every submodular function $\rho$.
 An \textit{(integral) base polyhedron} $B(\rho)$ associated with a submodular function $\rho$ is
the set of integral vectors given by
\[
B(\rho) = \{x\in \Z^n \mid x(Y) \le \rho(Y)\ (Y \in 2^N),\ x(N) = \rho(N)\}.
\]
 An integral base polyhedron is also referred to as an M-convex set \cite{Murota03book}.

 A {polymatroid} is defined by a submodular function 
$\rho: 2^N \to \Z_+$ such that function values are non-negative integers and 
$\rho$ is monotone non-decreasing (i.e., $\rho(X) \le \rho(Y)$ whenever $X \subseteq Y$);
such a function is called a \textit{polymatroid rank function}.
 An \textit{(integral) polymatroid} associated with a polymatroid rank function $\rho$
is the set of non-negative integral vectors given as
\[
P(\rho) = \{x\in \Z_+^n \mid x(Y) \le \rho(Y)\ (Y \in 2^N)\}.
\]
 A vector $x \in P(\rho)$ is called a base of the polymatroid $P(\rho)$
if it satisfies $x(N)= \rho(N)$.

\paragraph{M-convex and \Mnat-convex function}

 For a function $f: \Z^n \to \Rinf$  defined on
integer lattice points, we define the effective domain of $f$
by  $\dom f= \{x \in \Z^n \mid f(x) < + \infty\}$.
 A function $f$ with $\dom f  \ne \emptyset$  is said to be \textit{M-convex} 
if it satisfies the following exchange axiom:
\begin{quote}
{\bf (M-EXC)}
for every
$x, y \in \dom f$ and $i \in \suppp(x-y)$,
there exists some
$j \in \suppm(x-y)$ such that
\[
 f(x)  + f(y) \ge f(x - \chi_i + \chi_j) + f(y + \chi_i - \chi_j),
\]
\end{quote}
where 
\[
\suppp(x-y) = \{i \in N \mid x(i)> y(i)\},
\quad
\suppm(x-y) = \{i \in N \mid x(i) < y(i)\}.
\]

 We say that $f$ is \textit{M-concave} if $-f$ is an M-convex function.
The concept of M-convex function is an extension of valuated matroid due to Dress
and Wenzel \cite{DW90}; 
for a function $f$ with $\dom f \subseteq \{0,1\}^n$,
$f$ is a valuated matroid if and only if it is M-concave.

 The concept of M-convexity is deeply related with integral base polyhedra.
 For an M-convex function $f:\Z^n \to \Rinf$, 
the effective domain $\dom f$ is an integral base polyhedron;
in addition, the set of minimizers $\arg\min f$ 
is an integral base polyhedra if $\arg\min f \ne \emptyset$.
 On the other hand, given an integral base polyhedron $B \subseteq \Z^n$, the
 indicator  function $\delta_B: \Z^n \to \{0, + \infty\}$ defined by 
\[
\delta_B(x) = 
\begin{cases}
0 & (\mbox{if }x \in B),\\
+ \infty &(\mbox{otherwise})
\end{cases}
\]
is an M-convex function.

 An \Mnat-convex function is a variant of M-convex function.
 The exchange axiom (M-EXC) implies that $x(N)=y(N)$ for every $x,y \in \dom f$, i.e.,
the effective domain $\dom f$ is contained in a hyperplane
of the form $x(N)=r$ with some integer $r$.
 This fact naturally leads us to consider the projection of an M-convex function with $n$ variables
to a function with $n-1$ variables; we call such a function an 
\textit{\Mnat-convex} function~\cite{MS99}. 
 More precisely, a function $f: \Z^n \to \Rinf$ is said to be \emph{\Mnat-convex}
if the function $\tilde{f}: \Z^{\tilde{N}} \to \Rinf$ given as
\begin{align}
\label{eqn:tilde_f}
 \tilde{f}(x, x_0) = 
\begin{cases}
f(x)   & (\mbox{if } x(N) + x_0 = 0),\\
+ \infty & (\mbox{otherwise})
\end{cases}
\qquad ((x, x_0) \in \Z^{\tilde{N}} = \Z^n \times \Z)
\end{align}
is an M-convex function, where $\tilde{N} = N \cup\{0\}$.

 An \Mnat-convex function can be characterized by the following exchange axiom:
\begin{quote}
{\bf (\Mnat-EXC)}
for every
$x, y \in \dom f$ and $i \in \suppp(x-y)$,
there exists some
$j \in \suppm(x-y)\cup\{0\}$ such that
\[
 f(x)  + f(y) \ge f(x - \chi_i + \chi_j) + f(y + \chi_i - \chi_j).
\]
\end{quote}
 Note that $j$ can be equal to $0$ in (\Mnat-EXC), which is not possible in (M-EXC).
  Hence, the class of \Mnat-convex functions properly contains that of
M-convex functions, while they are essentially equivalent 
by the definition of \Mnat-convexity.

Gross-substitutes valuation \cite{KC82} in mathematical economics 
and its generalization
called strong-substitutes valuation \cite{Milgrom-Strulovici2009} 
are essentially equivalent to the concept of \Mnat-convexity 
(see Example \ref{ex:subst} for details).
 Note also that the indicator function $\delta_{P}$ of an integral polymatroid $P$
is an \Mnat-convex  function.

\paragraph{Examples}

 We present some examples of M-convex and \Mnat-convex functions.

\begin{example}[Minimum cost flow]\rm
\label{ex:mcf}
 M-convex functions arise from the minimum cost flow problem.
 Let $G = (V, A)$ be a directed graph with 
two disjoint vertex subsets $S,T \subseteq V$,
where $S$ (resp., $T$) represents the set of source (resp., sink) vertices.
 For each arc $a \in A$ we are given 
an arc capacity $u(a) \in \Z_+$.
 A vector $\xi \in \Z^A$ is called a flow, and
the boundary $\partial \xi \in \Z^V$ of a flow $\xi$ 
is given by
\[
\partial \xi(v) = \sum\{\xi(a) \mid
\text{arc $a$ leaves $v$}\}
  - \sum\{\xi(a) \mid \text{arc $a$ enters $v$}\}
\qquad (v \in V).
\]
 A flow $\xi \in \Z^A$ is said to be feasible if it satisfies
\[
0 \le \xi(a) \le u(a) \ (a \in A), \qquad
\partial \xi(v) = 0 \ (v \in V \setminus (S \cup T)).
\]

 Suppose that we are also given a univariate convex functions $f_a: \Z\to \R$
for each $a \in A$,  which represents the cost of flow on arc $a$.
 The minimum cost of a flow that realizes supply/demand values 
$x \in \Z^{S \cup T}$ is represented by
a function $f: \Z^{S \cup T} \rightarrow \Rpminf$ defined as
\[
 f(x) = \inf_{\xi\in \Z^A}
\left\{\left.\sum_{a \in A}f_a(\xi(a)) \ \right| \ 
\text{$\xi$ is a feasible flow  with
$(\partial \xi)(v) = x(v)$ $(v \in S\cup T)$ }
\right\};
\]
we define $f(x) = + \infty$
if there exists no feasible flow $\xi$ 
with $(\partial \xi)(v) = x(v)\ (v \in S\cup T)$.
 It can be shown that $f$ is an M-convex function, provided that
$f(x) > - \infty$ for $x \in \Z^n$ \cite{Murota96,Murota98}.

 An optimal solution of the minimum cost flow problem can be obtained 
by a variant of the successive shortest path algorithm:
starting with the zero flow, we repeatedly select a shortest path
among all paths from sources to sinks in the auxiliary network
and  augment a flow by one unit along the selected path.
 This algorithm can be seen as a specialized implementation
of the steepest descent algorithm for M-convex function minimization 
to the function $f$ mentioned above.
\qed
\end{example}

\begin{example}[Resource allocation]
\label{ex:rap}
\rm
 The {\it (separable convex) resource allocation problem under a
polymatroid constraint} is given as follows  \cite{FG86,Groenevelt91,Hochbaum94}:
\[
\mbox{RAP:} \quad
\mbox{Minimize }  \textstyle \sum_{i =1}^n f_i(x(i))
\quad  \mbox{subject to } 
\displaystyle
x \in P(\rho),\ x(N)= \rho(N),
\]
where $f_i: \Z \to \R$ $(i \in N)$ is a univariate convex function
and $\rho: 2^N \to \Z_+$ is a polymatroid rank function;
 see \cite{Fujishige05,IK88,KSI13} for review of RAP.
 RAP can be regarded as a special case of M-convex function minimization
since  function $f_{\rm RAP1}: \Z^n \to \Rinf$ defined by
\[
f_{\rm RAP1}(x) =  \left\{
\begin{array}{ll}
\sum_{i =1}^n f_i(x(i))
& (\mbox{if } x\in \Z^n\mbox{ is a feasible solution to RAP}),\\
+ \infty & (\mbox{otherwise})
\end{array} 
\right.
\]
satisfies (M-EXC)
\cite{Murota96,Murota98}.
 We can also reformulate RAP as the minimization of 
function $f_{\rm RAP2}: \Z^n \to \Rinf$ defined by 
\[
f_{\rm RAP2}(x) =  \left\{
\begin{array}{ll}
\sum_{i =1}^n f_i(x(i))
& (\mbox{if } x\in P(\rho)),\\
+ \infty & (\mbox{otherwise})
\end{array} 
\right.
\]
under the constraint $x(N) = \rho(N)$.
It can be shown that $f_{\rm RAP2}$ is an \Mnat-convex function
(cf.~\cite{Murota96,Murota98,MS99}).
\qed
\end{example}

\begin{example}[Strong-substitutes valuations]\rm
\label{ex:subst}
 The concepts of gross-substitutes and strong-substitutes valuations
in mathematical economics
are known to be equivalent to \Mnat-concave functions \cite{FY03,ST15}.
 The \emph{gross-substitutes} condition for a single-unit valuation $f: \{0,1\}^n \to \R$,
introduced by Kelso and Crawford \cite{KC82} (see also \cite{GS99,GS00}),
is described using price vectors $p,q \in \R^n$ as follows:
\begin{quote}
 {\bf (GS)}  
$\forall p,q  \in \R^n$ with $p \leq q$,
$\forall x\in \arg\max_y\{f(y)-p^\top y\}$, $\exists x'\in \arg\max_y\{f(y)-q^\top y\}$:\\
 $x'(i)\ge x(i)$ for all $i\in N$ with $q(i)=p(i)$.
\end{quote}
 The gross-substitutes condition is extended to multi-unit valuation functions 
(i.e., functions defined on integer interval $0 \le x(i)\le u(i)$ $(i \in N)$)  
by Milgrom and Strulovici \cite{Milgrom-Strulovici2009}, 
which is called the \emph{strong-substitutes} condition.
 We say that a multi-unit valuation function $f$ satisfies
the strong-substitutes condition if $f$ satisfies (GS) when all units of every item is regarded as different items.
 Equivalence between gross-substitutes (and strong-substitutes) valuations
and \Mnat-concave functions is shown by Fujishige and Yang \cite{FY03} (see also \cite{ST15}).
 It is known that gross-substitutes (and strong-substitutes) valuation functions
enjoy various nice properties in mathematical economics.
 In particular, in the auction market with multiple indivisible items,
gross-substitutes condition for bidders' valuations  implies the existence of 
a Walrasian equilibrium. 

 We here consider the following fundamental problem in economics: given a strong-substitutes valuation $f$ and 
a price vector $p \in \R^n$ for items,
find an item set $x \in \dom f$ maximizing the value $f(x) - p^\top x$.
  This problem can be seen as a maximization of an \Mnat-concave function
since $f$ is an \Mnat-concave function and the class of \Mnat-concave functions is closed 
under the addition of linear functions.
\qed  
\end{example}


\section{Application to M-convex Function Minimization}
\label{sec:Mconv}

In this section, 
minimization of an M-convex function $f:\Z^n \to \Rinf$ is considered.  
 We first review the basic steepest descent algorithm, 
and then propose algorithms using long step length.
 Throughout this section, we assume the boundedness of $\dom f$;
this assumption implies $\arg\min f \ne \emptyset$, in particular.

\subsection{Review of Steepest Descent Algorithm}
\label{sec:M-review}

 The steepest descent algorithm for M-convex function minimization 
is based on the characterization of a minimizer
by local minimality condition.

\begin{theorem}[{\cite{Murota96,Murota98,Murota03book}}]
\label{thm:M-minimizer}
 For an M-convex function $f: \Z^n \to \Rinf$,
a vector $x^* \in \dom f$ 
is a minimizer of $f$ if and only if
$f(x^* + \chi_i - \chi_j) \ge f(x^*)$ for all $i, j \in N$.
\end{theorem}

\noindent
 In the algorithm description,
we use a vector of the form $+ \chi_i -\chi_j$ $(i, j \in N)$,
which is referred to as a \textit{direction} in this section.
 For $x \in \dom f$ and a direction $+ \chi_i -\chi_j$,
we denote 
\[
f'(x;i, j) = f(x + \chi_i -\chi_j) - f(x),
\]
i.e., ${f}'(x;i,j)$ is the slope of function $f$ at $x$
in the direction $+ \chi_i -\chi_j$.
Note that $f'(x;i,i)=0$ by definition.
 For $x \in \dom f$, we say that a direction
$+ \chi_i -\chi_j$ is a \textit{steepest descent direction} of $f$ at $x$
if it minimizes the value $f'(x;i, j)$ among all directions.
 We denote 
\[
 \phi(x) = \min_{i, j \in N}f'(x;i, j),
\]
i.e., $\phi(x)$ is the slope of a steepest descent direction at $x$.
 We have $\phi(x) \le 0$ for every $x \in \dom f$, 
and Theorem~\ref{thm:M-minimizer} implies that 
the equality holds if and only if $x$ is a minimizer of~$f$.

 We present below a basic version of the steepest descent algorithm.
 In the algorithm, the vector $x$ is
repeatedly moved in a steepest descent direction until $\phi(x) = 0$ holds.

\begin{flushleft}
 \textbf{Algorithm} {\sc M-SD} 
\\
 \textbf{Step 0:} 
 Let $x_0 \in \dom f$ be an arbitrarily chosen initial vector.
 Set $x:=x_0$.
\\
 \textbf{Step 1:} 
 Let $i, j \in N$ be elements that minimize $f'(x;i, j)$.
\\
 \textbf{Step 2:} 
If $f'(x;i, j) = 0$ then output $x$ and stop.
\\
 \textbf{Step 3:} 
  Set $x := x + \chi_i - \chi_j$, and go to Step 1.
\end{flushleft}

 It is easy to see that the $\ell_1$-distance from $x$ to (the nearest) minimizer of $f$ 
reduces \textit{at most} two in each iteration, which implies that 
the number of iterations in the algorithm {\sc M-SD} is at least
$(1/2)\min\{\|y - x_0 \|_1 \mid y \in \arg\min f\}$.
 It turns out that this bound is tight \cite{Shioura2022}.
 We denote by $\tau(x_0)$ the $\ell_1$-distance 
between a vector $x_0 \in \dom f$ and the set of minimizers $\arg\min f$, i.e.,
\[
 \tau(x_0) = \min\{\|y - x_0 \|_1 \mid y \in \arg\min f\}.
\]

\begin{theorem}[{\cite[Corollary 4.2]{Shioura2022}}]
\label{thm:M_sd_bound}
 For an  M-convex function $f:\Z^n \to \Rinf$ with $\arg\min f \ne \emptyset$,
the algorithm {\sc M-SD} outputs a minimizer $x^*$ of $f$
satisfying $\|x^* - x_0 \|_1= \tau(x_0)$,
and 
the number of iterations is exactly equal to $(1/2)\tau(x_0)$.
\end{theorem}

\subsection{Use of Long Step Length}
\label{sec:M-long-step}

 In the algorithm {\sc M-SD}, once a direction $+ \chi_i - \chi_j$
is selected, the  vector $x$ moves in the direction only by unit step length.
 We will modify  the algorithm
so that the current vector moves 
in the selected direction by multiple step length as far as the slope 
in the direction remains the same.
 
 In each iteration of the steepest descent algorithm,
the slope of $f$ in a steepest descent direction at $x$ (i.e., the value $\phi(x)$)
is non-decreasing. 

\begin{prop}[{cf.~\cite[Proposition 4.3]{Shioura2022}}]
\label{prop:M-mj-sdd-mono}
 Let $y \in \dom f$ be a vector with $\phi(y) < 0$,
and $i,j \in N$ be distinct elements such that
$f'(y;i,j) = \phi(y)$.
 Then, $\phi(y+ \chi_i - \chi_j) \ge \phi(y)$ holds.
 Moreover, for every distinct $h,k \in N$ it holds that 
$f'(y+ \chi_i - \chi_j;h,k) \ge \phi(y)$,
and if the inequality holds with equality, then 
$+ \chi_h - \chi_k$ is a steepest descent direction at $y+ \chi_i - \chi_j$.
\end{prop}
\noindent
 For readers' convenience, we provide a proof of Proposition \ref{prop:M-mj-sdd-mono}
in Section \ref{proof:prop:M-mj-sdd-mono} of Appendix.

 This monotonicity implies that a steepest descent direction 
 chosen in Step 1 of {\sc M-SD} can be used again in the next iteration
if the slope in the direction remains the same.
 Based on this observation, we modify the algorithm as follows:
 once a steepest descent direction $+ \chi_i - \chi_j$ is selected in Step 1,
the vector $x$ is updated to $x+\bar{c}(x;i,j)(\chi_i- \chi_j)$
with the step length $\bar{c}(x;i, j)$ given  by 
\begin{align}
\label{eqn:def:M-bar_c}
\bar{c}(x;i, j) = \max\{\lambda \in \Z_+ \mid 
 f(x + \lambda(\chi_i - \chi_j))- f(x) = \lambda f'(x;i, j)   \}.
\end{align}
 This idea can be incorporated in the algorithm {\sc M-SD} 
by replacing Step 3 with the following:

\begin{flushleft}
  \textbf{Step 3:} 
  Set $x:=x+\bar{c}(x;i,j)(\chi_i- \chi_j)$ and go to Step 1.
\end{flushleft}

  The modified algorithm, denoted as {\sc M-LSD}, can be seen as a special implementation of 
the basic algorithm {\sc M-SD}, and therefore the theoretical time bound
in Theorem \ref{thm:M_sd_bound} can be also applied to {\sc M-LSD} as well.
 While it is expected that {\sc M-LSD} runs faster 
than {\sc M-SD} in practice,
no better theoretical time bound for {\sc M-LSD} is known so far.

To obtain an alternative theoretical bound for a long-step version of the steepest
descent algorithm, 
we further modify the algorithm
by selecting steepest descent directions in some specific order.
At the termination of the modified algorithm {\sc M-LSD2}, 
the output vector $x$ satisfies 
$\phi(x)=0$, and therefore it is a minimizer of $f$ by Theorem \ref{thm:M-minimizer}.

\begin{flushleft}
 \textbf{Algorithm} {\sc M-LSD2} 
\\
 \textbf{Step 0:} 
 Let $x_0 \in \dom f$ be an arbitrarily chosen initial vector.
 Set $x:=x_0$.
\\
 \textbf{Step 1:} 
 If $\phi(x) = 0$, then output $x$ and stop.
\\
 \textbf{Step 2:} 
 Let $x'$ be the output of the procedure {\sc M-IncSlope}($x$). 
Set $x:=x'$.
Go to Step 1.
\end{flushleft}

  Given a vector $x \in \dom f$, the procedure {\sc M-IncSlope}($x$)
initially sets the vector $y$ to $x$, 
repeatedly updates $y$ 
by using steepest descent directions with slope equal to $\phi(x)$,
and finally obtain $y$ with $\phi(y) > \phi(x)$.
 The outer iteration of the procedure consists of Steps 1 and~2, 
and there is an inner iteration in Step 1. 
 In Step 1 of each outer iteration,
we select any $i \in N$ first.
 Then, for each $j \in N \setminus \{i\}$,
we check whether the slope $f'(y;i,j)$ is equal to $\phi(x)$;
if it is true, $+ \chi_i - \chi_j$ is a steepest descent direction by
Proposition \ref{prop:M-mj-sdd-mono},
and the vector $y$ is updated to
$y+\bar{c}(y;i,j)(\chi_i - \chi_j)$.

\begin{flushleft}
 \textbf{Procedure} {\sc M-IncSlope}($x$) 
\\
 \textbf{Step 0:} 
Set $y := x$ and $N^+:=N$.
\\
 \textbf{Step 1:} 
Take any $i \in N^+$, set $N^-_i := N \setminus\{i\}$, 
and do the following steps.
\\
\quad 
 \textbf{Step 1-1:} 
 Take any $j \in N^-_i$.
 If $f'(y;i,j) = \phi(x)$, set $y:= y+\bar{c}(y;i,j)(\chi_i - \chi_j)$.
\\
\quad 
 \textbf{Step 1-2:} 
 Set $N^-_i := N^-_i  \setminus\{j\}$.  
If $N^-_i = \emptyset$, then go to Step 2.
Otherwise, go to Step 1-1.
\\
 \textbf{Step 2:} 
 Set $N^+ := N^+ \setminus\{i\}$. 
 If $N^+ = \emptyset$, then output $y$ and stop.
Otherwise, go to Step 1.
\end{flushleft}

 The next theorem shows that $\phi(x)$ increases strictly in each iteration of 
{\sc  M-LSD2}. 
 Proof  is given in Section  \ref{sec:proofs-M}.

\begin{theorem}
\label{thm:M-inc_min_slope}
 For a vector $x \in \dom f$ with $\phi(x) < 0$,
the output $x'$ of the procedure {\sc M-IncSlope}$(x)$
satisfies 
$\phi(x')  > \phi(x)$.
\end{theorem}

 We analyze the running time of the algorithm {\sc M-LSD2}. 
 In the following, we assume that 
$f$ is an integer-valued function is given as an evaluation oracle that
requires constant time for function value evaluation.

 By Theorem \ref{thm:M-inc_min_slope},
the algorithm {\sc  M-LSD2} 
terminates in at most $|\phi(x_0)|$ iterations if $f$ is an integer-valued
function.
 The step length $\bar{c}(y;i,j)$ 
for  $y \in \dom f$ and $i, j \in N$ can be 
computed in $\Oh(\log L_\infty)$ time by binary search, where
\begin{align}
\label{eqn:diameter}
   L_\infty = \max\{\|y-y'\|_\infty \mid y, y' \in \dom f\},
\end{align}
which is the ``diameter'' of $\dom f$.
The step length $\bar{c}(y;i,j)$ is computed once
for every pair of distinct $i, j \in N$
in the procedure {\sc M-IncSlope}$(x)$. 
 Therefore, {\sc M-IncSlope}$(x)$ runs in
$\Oh(n^2 \log L_\infty)$ time.
 The discussion above, together with Theorem \ref{thm:M_sd_bound}, 
implies the following time bound for {\sc M-LSD2}.

\begin{theorem}
\label{thm:M-runtime}
 Let
$f: \Z^n \to \Z\cup \{+ \infty\}$ be an integer-valued M-convex function 
with bounded $\dom f$,
and assume that the function value of $f$ can be evaluated in constant time.
 Then, the algorithm {\sc M-LSD2} outputs a minimizer of $f$
in
$\Oh(n^2 (\log L_\infty) \cdot \min\{ |\phi(x_0)|, \tau(x_0)\})$ time.
\end{theorem}


\subsection{Proof of Theorem \ref{thm:M-inc_min_slope}}
\label{sec:proofs-M}

 Let $x \in \dom f$ be a vector  with $\phi(x) < 0$,
and $x'$ be the output of the procedure {\sc M-IncSlope}$(x)$.
 The goal of this section is to prove the inequality $\phi(x')  > \phi(x)$. 

 Let us consider Step 1 in some outer iteration
of the procedure {\sc M-IncSlope}$(x)$, and
let $i\in N$ be the element taken at the beginning of Step 1. 
By the definition of the step length $\bar{c}(y;i,j)$,
vector $y$ at the end of Step 1-1 
satisfies the inequality $f'({y};i,j) > \phi(x)\ (=\phi(y))$.
 We first show that this inequality is preserved until the end
of the inner iterations in Step 1.

\begin{lemma}
\label{lem:M-slope_inc_1-2}
Let $y \in \dom f$ be a vector with $\phi(y)=\phi(x)$,
and $i,j \in N$ be distinct elements such that $f'(y; i, j) > \phi(x)$.
 For $k \in N \setminus\{i,j\}$ with $y+ \chi_i - \chi_k \in \dom f$, we have
$f'(y+ \chi_i - \chi_k;  i, j) \ge f'(y; i, j) > \phi(x)$.
\end{lemma}
 
\begin{proof}
 Let $\tilde{y} =y + \chi_i - \chi_k + \chi_i - \chi_j$.
 It suffice to show that $f(\tilde{y}) - f(y+ \chi_i - \chi_k) \ge f'(y; i, j)$.
since $f'(y; i, j) > \phi(x)$.
 If $f(\tilde{y}) = + \infty$ then we are done; hence we assume $\tilde{y} \in \dom f$.
 By (M-EXC) applied to $y$, $\tilde{y}$, and $j \in \suppp(y- \tilde{y})$, it holds that
\begin{align*}
f(y) + f(\tilde{y}) 
& \ge  f(y - \chi_j + \chi_i) + f(\tilde{y} + \chi_j - \chi_i) 
=  f(y - \chi_j + \chi_i) + f(y + \chi_i - \chi_k)
\end{align*}
since $\suppm(y- \tilde{y})=\{i\}$.
 It follows that
\[
f(\tilde{y}) - f(y+ \chi_i - \chi_k)
 \ge f(y + \chi_i - \chi_j) - f(y) = f'(y;i,j).
\]
\end{proof}

  Repeated application of Lemma~\ref{lem:M-slope_inc_1-2} implies that
vector $y$ at the end of Step 1 satisfies the inequalities
\begin{align}
\label{eqn:M_slope_inc_1:1}
f'(y; i, j) > \phi(x) \qquad (j \in N\setminus\{i\}).
\end{align}

 Suppose that the inequalities \eqref{eqn:M_slope_inc_1:1} for some $i \in N$
is satisfied by the vector $y$ at the end of Step 1 in some outer iteration.
 We then show that these  inequalities are preserved
in the following outer iterations.

\begin{lemma}
\label{lem:M-slope_inc_2}
Let $y \in \dom f$ be vectors with $\phi(y)=\phi(x)$,
and $i \in N$ be an element satisfying $f'(y; i, j) > \phi(x)$ for every $j \in N\setminus\{i\}$.
Also, let $h, k \in N$ be distinct elements such that
$f'(y;h,k) = \phi(x)$, i.e., $+\chi_h - \chi_k$ is a steepest descent direction at $y$.
 Then, 
$f'(y+\chi_h - \chi_k; i, j) > \phi(x)$ holds for every $j \in N\setminus\{i\}$.
\end{lemma}

\begin{proof}
 We fix $j^* \in N\setminus \{i\}$ and 
denote $\tilde{y}=y+\chi_h - \chi_k + \chi_i - \chi_{j^*}$. 
 It suffices to show that 
\begin{align}
  \label{eqn:lem:M-slope_inc_2-1}
f(\tilde{y}) - f(y+\chi_h - \chi_k) > \phi(x).
\end{align}
 If $f(\tilde{y}) = + \infty$ then we are done; hence we assume $\tilde{y} \in \dom f$.

We first consider the case with $i\ne k$ and $j^*\ne h$. 
  Since $\suppp(y - \tilde{y})=\{k,{j^*}\}$,
 (M-EXC) applied to $y$, $\tilde{y}$, and $k \in \suppp(y- \tilde{y})$ implies that 
\begin{align*}
f(y) + f(\tilde{y}) 
& \ge \min\{f(y  + \chi_h - \chi_k) + f(y+ \chi_i - \chi_{j^*}), 
f(y+ \chi_{i} - \chi_k  ) + f(y+ \chi_h - \chi_{j^*})\}\\
& \ge f(y  + \chi_h - \chi_k) + \min\{f(y+ \chi_i - \chi_{j^*}), f(y + \chi_i- \chi_k)\}\\
& > f(y + \chi_h  - \chi_k) + f(y)+\phi(x),
\end{align*}
where the second inequality is by the assumption that
$+\chi_h - \chi_k$ is a steepest descent direction at $y$,
and the last inequality is by 
$f'(y; i, j) > \phi(x)$ $(j \in N\setminus\{i\})$.
 Hence, \eqref{eqn:lem:M-slope_inc_2-1} follows.

 We next consider the case where $i=k$ or $j^* = h$ holds.
 Then, we have $\tilde{y} = y+\chi_{h'}-\chi_{k'}$
for some $h', k' \in N$, where it is possible that $h'=k'$.
 Since $+\chi_h - \chi_k$ is a steepest descent direction at $y$, 
it holds that $f'(y;h',k') \ge  f'(y;h,k)$.
 Therefore, we obtain the inequality \eqref{eqn:lem:M-slope_inc_2-1} as follows:
\begin{align*}
f(\tilde{y}) - f(y+\chi_h - \chi_k) 
&= f({y}+\chi_{h'}-\chi_{k'}) - f(y+\chi_h - \chi_k) \\
& = f'(y;h',k')  - f'(y;h,k) \ge 0 > \phi(x).
\end{align*}
\end{proof}

 By repeated application of Lemma \ref{lem:M-slope_inc_2}, we obtain
the inequalities $f'(x'; i, j) > \phi(x)$ 
$(i,j \in N,\ i\ne j)$ 
for the vector $x'$ at the end of the procedure {\sc M-IncSlope}$(x)$,
implying the desired inequality $\phi(x') > \phi(x)$.

\subsection{Some Remarks}
 
\begin{remark}\rm
  The procedure {\sc M-IncSlope}$(x)$
uses each direction $+ \chi_i - \chi_j$ at most once for update of the current vector.
 Therefore, Theorem \ref{thm:M-inc_min_slope} implies that
if some direction is used twice in the algorithm {\sc M-LSD2},
then its slope in that direction must be  different. 
 In contrast, it can happen that the basic algorithm {\sc M-LSD} uses the same direction
twice (or more) but the slope remains the same, as shown in the following example.

 Let us consider the function $f: \Z^4 \to \Rinf$ given as
\begin{align*}
& \textstyle
\dom f = \{x \in \Z_+^4 \mid \sum_{i=1}^4 x(i)=3,\ x(1) \le 2,\ x(2)\le 2,\ 
x(3) \le 1,\ x(4)\le 1\} \setminus\{(0,2,1,0)\},
\\
&
f(x) = - x(1)-x(3) \quad (x \in \dom f\setminus\{(2,0,0,1)\}),
\qquad
f(2,0,0,1) = -1.
\end{align*}
 This is an M-convex function; indeed, we can show this by checking (M-EXC)
 for $f$.

 Suppose that the algorithm {\sc M-LSD} is applied to function $f$
with the initial vector $x_0 = (0,2,0,1)$.
 A possible trajectory of the vector $x$ generated by M-LSD is
\[
(0,2,0,1) \to (1,1,0,1) \to (1,1,1,0) \to (2,0,1,0),
\]
where the direction $+\chi_1 - \chi_2=(+1,-1,0,0)$ 
is used in the first and the third iterations, and its slope is equal to $-1$ in both of the iterations.

 We then apply the algorithm {\sc M-LSD2} with the same initial vector $x_0$,
and select $i=1$ in Step 1 of the first iteration.
 Then, a possible trajectory of the vector $x$ is
\[
(0,2,0,1) \to (1,1,0,1) \to (2,1,0,0) \to (2,0,1,0),
\]
where different directions are used in each iteration.
\qed  
\end{remark}

\begin{remark}\rm
 We have shown that 
the number of iterations required by {\sc M-LSD2} is at most $|\phi(x_0)|$
if $f$ is integer-valued.
 We can also show that
the number of iterations is bounded by
$\sqrt{2(f(x_0) - \min f)}$
with $\min f = \min\{f(x) \mid x \in \dom f\}$.

 Let $k$ be the number of iterations required by {\sc M-LSD2},
and for $h=1,2,\ldots, k$, 
 let $x_h \in \dom f$ be vector $x$ at the end of the $h$-th iteration
of the algorithm.
 Then, $x_k$ is the output of the algorithm, which satisfies $\phi(x_k)=0$.

 We show that $(1/2) k^2 \le f(x_0) - \min f$ holds.
 The values $\phi(x_h)$ $(h=0,1,\ldots, k)$ are integers by assumption,
and strictly increasing by Theorem \ref{thm:M-inc_min_slope}.
 This fact, together with $\phi(x_k)=0$, implies that
$|\phi(x_h)| \ge k-h \ (h=0,1,\ldots, k)$.
 It follows that
\[
f(x_{h-1}) - f(x_{h}) 
 \ge |\phi(x_{h-1})| \ge k-h+1 \quad (h=1,2,\ldots, k).
\]
 Hence, we have
\begin{align*}
(1/2) k^2
 \le \sum_{h=1}^k (k-h+1)  
& \le \sum_{h=1}^k [f(x_{h-1}) - f(x_{h})]  
 = f(x_{0}) - f(x_{k}) = f(x_0) - \min f.
\end{align*}
The inequality $k \le \sqrt{2(f(x_0) - \min f)}$ follows immediately
from this.
\qed
\end{remark}

\begin{remark}\rm
 Theorem \ref{thm:M_sd_bound} implies monotonicity of vector $x$
in the basic steepest descent algorithm {\sc M-SD}.
\begin{prop}
\label{prop:M-sda-monotonicity}
 Let $x^*$ be the output of the algorithm {\sc M-SD}.
 In each iteration of the algorithm, component $x(i)$ $(i \in N)$ of
vector $x$ is non-increasing if $x^*(i) \le x_0(i)$ and
non-decreasing if $x^*(i) \ge x_0(i)$.
\end{prop}
\noindent
Indeed, if the statement does not hold,
then the number of iterations required by {\sc M-SD} must be strictly larger than
$(1/2)\|x^* - x_0\|_1$, a contradiction.
 By using this monotonicity, we can speed up the algorithms practically.
 
Proposition \ref{prop:M-sda-monotonicity} implies that
if some component $x(i)$ is increased (resp., decreased) in some iteration, then
it is never decreased (resp., increased) in the following iterations.
 Therefore, it possible to
restrict the choice of $i, j \in N$ in Step 1 of {\sc M-SD} as follows,
which may reduce the running time of the steepest descent algorithm.

\begin{flushleft}
 \textbf{Algorithm} {\sc M-SD$'$} 
\\
 \textbf{Step 0:} 
 Let $x_0 \in \dom f$ be an arbitrarily chosen initial vector.
 Set $x:=x_0$, $N_+:=N$, 
\\
\phantom{\textbf{Step 0:}}
and $N_-:=N$.
\\
 \textbf{Step 1:} 
 Let $i \in N_+$ and $j \in N_-$ be elements that minimize $f'(x;i, j)$.
\\
 \textbf{Step 2:} 
If $f'(x;i, j) \ge 0$ then output $x$ and stop.
\\
 \textbf{Step 3:} 
  Set $x := x + \chi_i - \chi_j$, $N_+:=N_+ \setminus\{j\}$, $N_-:=N_- \setminus\{i\}$,
and go to Step 1.
\end{flushleft}

 In a similar way, we can also modify the procedure {\sc M-IncSlope}$(x)$ as follows,
where a new index set $N^-$ is used in addition to $N^+$
and $N^-_i$ $(i \in N)$. 

 \begin{flushleft}
 \textbf{Procedure} {\sc M-IncSlope$'$}($x$) 
\\
 \textbf{Step 0:} 
Set $y := x$, $N^+:=N$, and $N^-:=N$.
\\
 \textbf{Step 1:} 
Take some $i \in N^+$, set $N^-_i := N^- \setminus\{i\}$, 
and do the following steps.
\\
\quad 
 \textbf{Step 1-1:} 
 Take some $j \in N^-_i$. If $f'(y;i,j) = \phi(x)$, then
set $y :=y+\bar{c}(y;i,j)(\chi_i - \chi_j)$,
\\
\quad\phantom{\textbf{Step 1-1:}}
$N^+ :=N^+ \setminus \{j\}$, and  $N^- :=N^- \setminus \{i\}$.
\\
\quad 
 \textbf{Step 1-2:} 
 Set $N^-_i := N^-_i \setminus\{j\}$.
 If $N^-_i = \emptyset$, then go to Step 2.
Otherwise, go to Step 1-1.
\\
 \textbf{Step 2:} 
 Set $N^+ := N^+ \setminus\{i\}$. 
 If $N^+ = \emptyset$, then output $y$ and stop.
Otherwise, go to Step 1.
\qed  
\end{flushleft}
\end{remark}


\section{Application to Constrained M-convex Function Minimization}
\label{sec:constM}

 The problem discussed in this section is formulated as follows: 
\begin{center}
Min$(f, R, k)$: \quad Minimize $f(x)$ \quad
subject to $x(R) = k$, $x \in \dom f$, 
\end{center}
\noindent
where $f: \Z^n \to \Rinf$ is an M-convex function with bounded $\dom f$,
$R$ is a non-empty subset of $N$, and $k \in \Z$. 
 It is known that this constrained  problem can be solved by a variant of
steepest descent algorithm \cite{Takazawa2023}, to which the idea of long step length
can be naturally applied as well.

The problem Min$(f, R, k)$ with $R=N$ is nothing but
an unconstrained minimization of $f$ since $x(N)$ is a constant for every $x \in \dom f$
(see Section \ref{sec:prelim}). 
 Hence, $R$ is assumed to be a proper subset of $N$
 (i.e., $\emptyset \ne R \subsetneq N$).
 In addition, feasibility of the problem  Min$(f, R, k)$ is assumed in the following;
Min$(f, R, k)$ is feasible
if and only if
$k$ satisfies $\underline{k} \le k \le \overline{k}$ with
$\underline{k} = \min\{x(R) \mid x \in \dom f\}$
and $\overline{k} = \max\{x(R) \mid x \in \dom f\}$
(cf.~\cite[Lemma 2]{Takazawa2023}).

\subsection{Review of Steepest Descent Algorithm}
\label{sec:constM-main}

 Given an optimal solution $x$ of Min$(f, R, k)$,
an optimal solution of Min$(f, R, k+1)$ can be obtained easily by using
a steepest descent direction at $x$.
 For every $k$ with $\underline{k} \le k \le \overline{k}$,
we denote by $M(k)$ and $z(k)$ the
set of optimal solutions and the optimal value for the problem Min$(f, R, k)$, i.e.,
\begin{align*}
& 
M(k) = \arg\min\{f(x) \mid x(R)=k,\ x \in \dom f\},
\\
& z(k) = \min\{f(x) \mid x(R)=k,\ x \in \dom f\}.
\end{align*}
 As in Section \ref{sec:Mconv},
we denote 
$f'(x;i, j) = f(x + \chi_i -\chi_j) - f(x)$
for $x \in \dom f$ and a direction $+ \chi_i -\chi_j$.
 We also define
\[
 \phi^R(x) = \min_{i \in R,\; j \in N\setminus R}f'(x;i, j) \qquad (x \in \dom f).
\]

\begin{prop}[{cf.~\cite[Lemma 3]{Takazawa2023}}]
\label{prop:constM-sdd-opt}
 Let $k$ be an integer with $\underline{k} \le k < \overline{k}$,
and $x \in M(k)$.
 Suppose that $i \in R$ and  $j \in N\setminus R$ 
minimize the value $f'(x;i, j)$, i.e., $f'(x;i, j) = \phi^R(x)$.
 Then, 
$x + \chi_i - \chi_j \in M(k+1)$ holds.
\end{prop}

 Repeated application of Proposition \ref{prop:constM-sdd-opt} implies that
an optimal solution of Min$(f, R, k)$
can be obtained by the following steepest descent algorithm \cite{Takazawa2023}.
 In the following, we assume that a vector in $M(\underline{k})$
is given in advance;
such a vector (i.e., an optimal solution of Min$(f, R, \underline{k})$) 
can be obtained by solving an unconstrained minimization of M-convex function 
$g(x) = f(x) - \Gamma \cdot x(R)$ $(x \in \Z^n)$
with a sufficiently large positive number $\Gamma$ (cf.~\cite{Takazawa2023}),
for which efficient algorithms are available.

\begin{flushleft}
 \textbf{Algorithm} {\sc ConstM-SD} 
\\
 \textbf{Step 0:} 
 Let $x_{\underline{k}} \in M(\underline{k})$ and 
set $x:=x_{\underline{k}}$.
 \\
 \textbf{Step 1:} 
If $x(R) = k$ then output $x$ and stop.
\\
 \textbf{Step 2:} 
 Let $i \in R$ and  $j \in N\setminus R$ be elements 
minimizing the value $f'(x;i, j)$.
\\
 \textbf{Step 3:} 
  Set $x := x + \chi_i - \chi_j$ and go to Step 1.
\end{flushleft}

  The algorithm outputs an optimal solution after
$k - \underline{k}$ iterations,
and each iteration requires $\Oh(|R|(n-|R|))$ time.
 Hence, we obtain the following result.

\begin{theorem}[{cf.~\cite[Theorem 5]{Takazawa2023}}]
Algorithm {\sc ConstM-SD} outputs 
an optimal solution  of {\rm Min$(f, R, k)$}
in $\Oh(|R|(n-|R|)(k - \underline{k}))$ time.
The running time reduces to
$\Oh(n(k - \underline{k}))$ 
if either $|R|$ or $|N \setminus R|$ is bounded by a constant.
\end{theorem}

\subsection{Use of Long Step Length}

 We then propose a long-step version of the algorithm {\sc ConstM-SD}.
  For this purpose, we show a monotonicity property of
steepest descent directions similar to Proposition~\ref{prop:M-mj-sdd-mono}.

\begin{prop}
\label{prop:constM-sdd-mono}
 Let $k$ be an integer with $\underline{k} \le k \le \overline{k}-2$ and
$x \in M(k)$.
 Also, let $i \in R$ and  $j \in N\setminus R$ be elements 
minimizing $f'(x;i, j)$.
 If $f'(x + \chi_i - \chi_j;i,j) =f'(x;i, j)$, then we have
\begin{align*}
&  f'(x + \chi_i - \chi_j;i,j) \le f'(x + \chi_i - \chi_j;h,\ell)\qquad  
(h \in R,\ \ell \in N\setminus R),\\
& x + 2 \chi_i - 2 \chi_j \in M(k+2).
\end{align*}
\end{prop}
\noindent
 Proof of this proposition 
is given in Section \ref{sec:constM-proofs}
in Appendix.

A long-step version of the algorithm {\sc ConstM-SD},
denoted as {\sc ConstM-LSD},
is obtained by replacing Step 3  in {\sc ConstM-SD} with the following,
where $\bar{c}(x;i, j)$ is given by \eqref{eqn:def:M-bar_c}.

\begin{flushleft}
  \textbf{Step 3:} 
  Set $\lambda: = \min\{k-x(R), \bar{c}(x;i, j)\}$,
$x := x + \lambda(\chi_i - \chi_j)$,
 and go to Step 1.
\end{flushleft}
\noindent
 Proposition \ref{prop:constM-sdd-mono} guarantees that
this modified algorithm can also find an optimal solution of {\rm Min$(f, R, k)$}.

\begin{theorem}
Algorithm {\sc ConstM-LSD} outputs 
an optimal solution  of {\rm Min$(f, R, k)$}.
\end{theorem}

 Although the number of iterations in the algorithm {\sc ConstM-LSD} 
can be the same as {\sc ConstM-SD} in the worst case,
it is expected that  {\sc ConstM-LSD} runs faster in practice.
 To obtain a better theoretical time bound,
we use algorithm {\sc ConstM-LSD2} and procedure 
{\sc ConstM-IncSlope}($x$), 
both of which are similar to (but different from) 
{\sc M-LSD2} and {\sc M-IncSlope}($x$)  
used for unconstrained M-convex function minimization.

\begin{flushleft}
 \textbf{Algorithm} {\sc ConstM-LSD2} 
\\
 \textbf{Step 0:} 
 Let $x_{\underline{k}} \in M(\underline{k})$ and 
set $x:=x_{\underline{k}}$.
\\
 \textbf{Step 1:} 
 If $x(R) = k$ then output $x$ and stop.
\\
 \textbf{Step 2:} 
 Let $x'$ be the output of the procedure {\sc ConstM-IncSlope}($x$). 
\\
\phantom{\textbf{Step 2:}}
 Set $x:=x'$ and go to Step 1.
\end{flushleft}

Input of the procedure {\sc ConstM-IncSlope}($x$) is
an optimal solution $x \in \dom f$ 
of the problem {\rm Min$(f, R, \hat{k})$} for some $\hat{k} < k$.
 The procedure initially sets the vector $y$ to $x$, 
then repeatedly updates $y$ by using $i \in R$ and $j \in N \setminus R$
with $f'(y;i,j) = \phi^R(x)$,
and finally obtain $y$ satisfying $\phi^R(y) > \phi^R(x)$
or $y(R)=k$ (or both).
 
 The outer iteration of the procedure consists of Steps 1 and 2, 
and there is an inner iteration in Step 1. 
 In Step 1 of each outer iteration,
we select any $i \in R$ first.
 Then, for each $j \in N \setminus R$
we check if the direction $+ \chi_i - \chi_j$ has the slope equal to $\phi^R(x)$;
if it is true, $y$ is updated to
the vector $y+\lambda(\chi_i - \chi_j)$ with $\lambda := \min\{k - y(R), \bar{c}(y;i,j)\}$.

\begin{flushleft}
 \textbf{Procedure} {\sc ConstM-IncSlope}($x$) 
\\
 \textbf{Step 0:} 
Set $y := x$ and $N^+:=R$.
\\
 \textbf{Step 1:} 
Take any $i \in N^+$, set $N^-_i := N\setminus R$, 
and do the following steps.
\\
\quad 
 \textbf{Step 1-1:} 
 Take any $j \in N^-_i$.
 If $f'(y;i,j) = \phi^R(x)$, set 
$\lambda := \min\{k - y(R), \bar{c}(y;i,j)\}$ and 
\\
\quad 
\phantom{\textbf{Step 1-1:}}
$y:= y+\lambda(\chi_i - \chi_j)$.
\\
\quad 
 \textbf{Step 1-2:} 
 Set $N^-_i := N^-_i \setminus\{j\}$.
 If $N^-_i = \emptyset$, then go to Step 2.
Otherwise, go to Step 1-1.
\\
 \textbf{Step 2:} 
 Set $N^+ := N^+ \setminus\{i\}$. 
 If $N^+ = \emptyset$, then output $y$ and stop.
Otherwise, go to Step 1.
\end{flushleft}


 Note that the vector $y$ at the end of Step 1-1 satisfies 
$f'(y;i,j) > \phi^R(x)$ if $y(R) < k$,
regardless of whether or not $y$ is updated in this step.
Using this inequality we can obtain 
the following property of {\sc ConstM-IncSlope}$(x)$,
which is similar to Theorem~\ref{thm:M-inc_min_slope};
proof is also similar and given in Section \ref{sec:proofs-constM}
in Appendix.

\begin{theorem}
\label{thm:constM-inc_min_slope}
 For a vector $x \in M(\hat{k})$ with an integer $\hat{k} < k$,
the output $x'$ of {\sc ConstM-IncSlope}$(x)$
satisfies $\phi^R(x')  > \phi^R(x)$, provided that $x'(R) < k$.
\end{theorem}

 By using Theorem \ref{thm:constM-inc_min_slope},
we can analyze the running time of the algorithm
{\sc ConstM-LSD2} in a similar way as in Section \ref{sec:Mconv}.
 The procedure {\sc ConstM-IncSlope}$(x)$ runs in
$\Oh(|R|(n-|R|) \log L_\infty)$ time with $L_\infty$ given by
\eqref{eqn:diameter}.
 By Theorem~\ref{thm:constM-inc_min_slope} 
and the equation $\phi^R(x) = z(k+1) - z(k)$ for $x \in M(k)$
(see Proposition \ref{prop:constM-sdd-opt}),
 the number of iterations required by the algorithm {\sc ConstM-LSD2}
is bounded by $\zeta(k)  -\zeta(\underline{k})$ with $\zeta(h) \equiv z(h+1) -z(h)$. 
 Also, the number of iterations is bounded by $k - \underline{k}$.
 Hence, we obtain the following result.

\begin{theorem}
\label{thm:constM-runtime}
 Algorithm {\sc ConstM-LSD2} outputs an optimal solution of
{\rm Min$(f, R, {k})$} in
$\Oh\big(|R|(n-|R|) (\log L_\infty)\min\big\{\zeta(k)  -\zeta(\underline{k}), k -
 \underline{k}\big\}\big)$ time.
The running time reduces to
$\Oh\big(n (\log L_\infty)\min\big\{\zeta(k)  -\zeta(\underline{k}), k -
 \underline{k}\big\}\big)$
if either $|R|$ or $|N \setminus R|$ is bounded by a constant.
\end{theorem}

\subsection{Constrained \Mnat-convex Function Minimization}
\label{sec:connect-const-Mnat}

The results for Min$(f,R,k)$ 
obtained in the previous subsection can be
extended to the constrained problem \Mnat-Min$(f,R,k)$
with an \Mnat-convex objective function $f$.
 In fact,  \Mnat-Min$(f,R,k)$ and  Min$(f,R,k)$
are essentially equivalent, as shown below,
and the results for Min$(f,R,k)$ can be easily translated 
in terms of \Mnat-Min$(f,R,k)$.

 Given an instance of  problem \Mnat-Min$(f,R,k)$,
let  $\tilde{f}: \Z^{\tilde{N}} \to \Rinf$ be the M-convex function
given by \eqref{eqn:tilde_f} with $\widetilde{N} = N \cup \{0\}$.
 By the definition of $\tilde{f}$, there exists a natural one-to-one
correspondence between vectors in $\dom f\ (\subseteq \Z^n)$ and 
those in $\dom \tilde{f}\ (\subseteq \Z^{\tilde{N}})$;
a vector $x \in \dom f$ corresponds to the vector 
$\tilde{x}\equiv   (x, -x(N)) \in \dom \tilde{f}$.
 Moreover, $x \in \dom f$ satisfies $x(R)=k$ 
if and only if $\tilde{x}$ satisfies $\tilde{x}(R)=k$
since $R\subseteq N$ and $0 \notin R$.
 Hence, $x \in \dom f$ is an optimal solution of \Mnat-Min$(f,R,k)$
if and only if 
$\tilde{x} \in \dom \tilde{f}$
is an optimal solution of Min$(\tilde{f},R,k)$.
 This observation shows that \Mnat-Min$(f,R,k)$
is equivalent to the constrained problem Min$(\tilde{f},R,k)$.

 Based on the relationship between the two problems,
algorithms for Min$(f,R,k)$  can be translated
in terms of \Mnat-Min$(f,R,k)$.
  The values $f'(x;i, j)$ and $\bar{c}(x;i, j)$ are defined
for $i,j \in N$ as in Section \ref{sec:constM-main}; 
in the case with $i \in N$ and $j = 0$, define
\begin{align*}
f'(x;i, 0) &= f(x + \chi_i) - f(x), \\
\bar{c}(x;i, 0) & = \max\{\lambda \in \Z_+ \mid 
 f(x + \lambda\chi_i)- f(x) = \lambda f'(x;i, 0)   \}.
\end{align*}
 Then, algorithm {\sc Const\Mnat-LSD}  for \Mnat-Min$(f,R,k)$
can be obtained from {\sc ConstM-LSD}  for Min$(f,R,k)$
as follows; the only difference is in Step 2, where 
the set $N\setminus R$ is replaced with $(N\setminus R)\cup\{0\}$.
 Note that
$\chi_0 = \0$ and $\underline{k} = \min\{x(R) \mid x \in \dom f\}$.

\begin{flushleft}
 \textbf{Algorithm} {\sc Const\Mnat-LSD} 
\\
 \textbf{Step 0:} 
 Let $x_{\underline{k}} \in M(\underline{k})$ and 
set $x:=x_{\underline{k}}$.
 \\
 \textbf{Step 1:} 
If $x(R) = k$ then output $x$ and stop.
\\
 \textbf{Step 2:} 
 Let $i \in R$ and  $j \in (N\setminus R)\cup\{0\}$ be elements 
minimizing the value $f'(x;i, j)$.
\\
 \textbf{Step 3:} 
  Set $\lambda: = \min\{k-x(R), \bar{c}(x;i, j)\}$,
$x := x + \lambda(\chi_i - \chi_j)$,
 and go to Step 1.
\end{flushleft}
  
\begin{theorem}
Algorithm {\sc Const\Mnat-LSD} outputs 
an optimal solution  of  {\rm \Mnat-Min$(f, R, k)$}
in $\Oh(|R|(n-|R|)(k - \underline{k}))$ time.
The running time reduces to $\Oh(n(k - \underline{k}))$ 
if either $|R|$ or $|N \setminus R|$ is bounded by a constant.
\end{theorem}

 Similarly, we can also obtain  algorithm {\sc Const\Mnat-LSD2} 
and procedure {\sc Const\Mnat-IncSlope}($x$) for \Mnat-Min$(f,R,k)$,
which are the same as {\sc ConstM-LSD2}  and {\sc ConstM-IncSlope}($x$)
 for Min$(f,R,k)$,
except that procedure {\sc Const\Mnat-IncSlope}($x$)
is used in {\sc Const\Mnat-LSD2} 
instead of {\sc ConstM-IncSlope}($x$),
and the set $N\setminus R$  in Step 1 of {\sc Const\Mnat-IncSlope}($x$)
is replaced with 
$(N\setminus R)\cup\{0\}$.

\begin{theorem}
\label{thm:constMnat-inc_min_slope}
 For a vector $x \in M(\hat{k})$ with an integer $\hat{k} < k$,
the output $x'$ of the procedure {\sc Const\Mnat-IncSlope}$(x)$
satisfies $\phi^R(x')  > \phi^R(x)$, provided that $x'(R) < k$.
\end{theorem}

\begin{theorem}
\label{thm:constMnat-runtime}
Algorithm {\sc Const\Mnat-LSD2} outputs an optimal solution of
{\rm \Mnat-Min$(f, R, {k})$} in
$\Oh\big(|R|(n-|R|) (\log L_\infty)
\min\big\{\zeta(k)  -\zeta(\underline{k}), k - \underline{k}\big\}\big)$ time.
The running time reduces to 
$\Oh\big(n (\log L_\infty)
\min\big\{\zeta(k)  -\zeta(\underline{k}), k - \underline{k}\big\}\big)$
if either of $|R|$ and $|N \setminus R|$ is bounded by a constant.
\end{theorem}

In the case of $R=N$, 
the description of the algorithm {\sc Const\Mnat-LSD} 
(and {\sc Const\Mnat-LSD2})
can be simplified as follows since
the element $j$ in the algorithm is always fixed to 0.

\begin{flushleft}
 \textbf{Algorithm} {\sc Const\Mnat-LSD3} 
\\
 \textbf{Step 0:} 
 Let $x_{\underline{k}} \in M(\underline{k})$ and 
set $x:=x_{\underline{k}}$.
\\
 \textbf{Step 1:} 
 If $x(N) = k$ then output $x$ and stop.
\\
 \textbf{Step 2:} 
 Let $i \in N$ be an element that minimizes $f'(x;i,0)$.
\\
 \textbf{Step 3:} 
 Set 
$\lambda := \min\{k - y(R), \bar{c}(y;i,0)\}$, 
  $x:= x+\lambda\chi_i$,
 and go to Step 1.
\end{flushleft}

\begin{coro}
Algorithm {\sc Const\Mnat-LSD3} outputs an optimal solution of
{\rm \Mnat-Min$(f, N, {k})$} in
$\Oh\big(n (\log L_\infty)
\min\big\{\zeta(k)  -\zeta(\underline{k}), k - \underline{k}\big\}\big)$ time.
\end{coro}

\begin{remark}\rm
The well-known greedy algorithm for linear optimization over a polymatroid 
 \cite{Edmonds70}
can be obtained as a specialized implementation of
the algorithm  {\sc Const\Mnat-LSD3}. 
 Note that in this case,
$\underline{k}=0$ holds, and 
$x_{\underline{k}}$ in Step 0 of {\sc Const\Mnat-LSD3}
is given as $x_{\underline{k}}= \0$.

 We can also specialize  {\sc Const\Mnat-LSD3} to 
the minimization of a separable-convex function over a polymatroid:
 \begin{center}
SC:   \quad Minimize \quad $\sum_{i \in N}f_i(x(i))$
\quad subject to \quad $x \in P(\rho),\ x(N) = \rho(N)$,
\end{center}
where $f_i: \Z \to \R$ is a univariate convex function for $i \in N$
and $\rho: 2^N \to \Z_+$ is a polymatroid rank function.
 The problem SC is formulated as \Mnat-Min$(f, N, \rho(N))$
 with an \Mnat-convex function $f$ such that
\[
 \textstyle 
 \dom f = P,\quad f(x) = \sum_{i \in N}f_i(x(i)) \ (x \in P).
\]
 By specializing {\sc Const\Mnat-LSD3} to SC,
we obtain the following algorithm,
which is a long step version of the incremental 
greedy algorithm  \cite{FG86,Groenevelt91}.

\begin{flushleft}
 \textbf{Algorithm} {\sc Greedy\_SC} 
\\
 \textbf{Step 0:} 
 Let $x =\0$.
 \\
 \textbf{Step 1:} 
  If $x(N) = \rho(N)$, then output $x$ and stop.
\\
 \textbf{Step 2:} 
 Let $i \in N$ be an element that minimizes $f_i(x(i)+1)-f_i(x(i))$
under the \\
\phantom{\textbf{Step 2:}}
  condition $x + \chi_i \in P$.
\\
 \textbf{Step 3:} 
 Let $\lambda \in \Z_+$ be the maximum integer satisfying 
$f_i(x(i)+\lambda) - f_i(x(i))=  $
\\
\phantom{\textbf{Step 2:}}
 $f_i(x(i)+1)-f_i(x(i))$ and $x + \lambda \chi_i \in P$.
  Set $x := x + \lambda \chi_i$
and go to Step 1.
\qed 
\end{flushleft}
\end{remark}


\section{Application to Polyhedral M-convex Function Minimization}
\label{sec:polyM}

  In this section, we consider the minimization of
a polyhedral M-convex functions defined on $\R^n$,
 and show that the steepest descent algorithms for M-convex functions on $\Z^n$
proposed in Section \ref{sec:M-long-step}
can be naturally extended to polyhedral M-convex functions.
 While the steepest descent algorithm finds a minimizer of a polyhedral M-convex function
if it terminates, it is not known so far whether the algorithm terminates
in a finite number of iterations.
 We show that
in a variant of the steepest descent algorithm with long step length,
the slope in the steepest descent direction increases strictly after $\Oh(n^2)$
iterations. 
 By using this property, we can obtain the first result 
on the finite termination of an exact algorithm for finding a global minimizer of a polyhedral M-convex functions. 
 
\subsection{Definition of Polyhedral M-convex Function}

 The concept of M-convexity is extended to polyhedral convex functions
on $\R^n$; a polyhedral convex function is a function $\R^n \to \Rinf$ such that
its epigraph $\{(x, \alpha) \in \R^n \times \R \mid f(x) \le \alpha\}$
is a polyhedron.
 By definition, a polyhedral convex function is a convex function on $\R^n$.

 A polyhedral convex function $\R^n \to \Rinf$ is said to be \textit{M-convex}
 \cite{polyML} 
if it satisfies the following  exchange axiom:
\begin{quote}
{\bf (M-EXC$[\R]$)}
for every
$x, y \in \domR f$ and $i \in \suppp(x-y)$,
there exists some
$j \in \suppm(x-y)$ and $\epsilon_0 > 0$ such that
\[
 f(x)  + f(y) \ge 
f(x - \epsilon(\chi_i - \chi_j)) + f(y + \epsilon(\chi_i - \chi_j))
\quad ( \epsilon \in [0,\epsilon_0]),
\]
\end{quote}
where $\domR f=\{x\in \R^n \mid f(x)< + \infty\}$.

 For a function $f: \Z^n \to \Rinf$
with bounded $\dom f$,
its convex closure $\bar{f}: \R^n \to \Rinf$ is given by
\[
\bar{f}(x) = \sup\{p^\top x + \alpha \mid 
p^\top y + \alpha \le f(y) \ (y \in \dom f)\} \qquad (x \in \R^n),
\]
which is a polyhedral convex function.
 It is known \cite{polyML} that if $f$ is an M-convex function, in addition,
then its convex closure $\bar{f}$ is polyhedral M-convex;
moreover, $\bar{f}(x) =f(x)$ holds for all $x \in \Z^n$
and $\min\{\bar{f}(x) \mid x\in \R^n\}=\min\{{f}(x) \mid x\in \Z^n\}$.
 In this sense, polyhedral M-convex functions are regarded as an extension of
 M-convex functions. 
 On the other hand, for a polyhedral M-convex function $f: \R^n \to \Rinf$,
its restriction on $\Z^n$ is an M-convex function on $\Z^n$ if 
$f$ is ``integral'' in the following sense:
$\arg\min\{f(x) - p^\top x \mid x \in \dom f\}$
is an integral polyhedron for every $p \in \R^n$.

 In Example \ref{ex:mcf} we provided an example of an M-convex function on $\Z^n$
arising from the minimum cost flow problem.
 In a similar way, we can obtain an example of polyhedral M-convex functions from
the minimum cost flow problem by 
replacing $f_a$ with a piecewise-linear convex function
and regarding $x$ and $\xi$ as real vectors
\cite{polyML}.

\subsection{Steepest Descent Algorithm}

 We propose a steepest descent algorithm for minimization
of a polyhedral M-convex function,
and show that it terminates after a finite number of iterations.

  Let $f:\R^n \to \Rinf$ be
a polyhedral  M-convex function such that $\domR f=\{x\in \R^n \mid f(x)< + \infty\}$
is bounded.
 This assumption guarantees the existence of a minimizer.
 It is well known that a global minimizer of an ordinary convex function
in real variables can be characterized by a local minimality
in terms of directional derivatives.
 For polyhedral M-convex functions, 
 local minimality is characterized by directional derivatives only in $\Oh(n^2)$ directions.
 For $x\in \domR f$ 
and $i,j \in N$, we denote by $f'_\R(x;i,j)$
the directional derivative of $f$ at $x$ in the direction $+ \chi_i - \chi_j$, i.e.,
\[
f'_\R(x;i,j) = \lim_{\alpha \downarrow 0}\frac{f(x+\alpha(\chi_i - \chi_j))-f(x)}{\alpha}.
\]
 Since $f$ is polyhedral convex,
$f'_\R(x;i,j)$ is well defined and 
there exists some $\epsilon > 0$ such that 
\[
f(x+\alpha(\chi_i - \chi_j)) = f(x)+\alpha f'_\R(x;i,j)  \qquad (0 \le \alpha \le \epsilon).
\]

\begin{theorem}[{\cite[Theorem~4.12]{polyML}}]
\label{thm:pM-minimizer}
 For a polyhedral M-convex function $f: \R^n \to \Rinf$,
a vector $x^* \in \domR f$ 
is a minimizer if and only if
$f'_\R(x^*;i,j) \ge 0$ for all $i, j \in N$.
\end{theorem}

 For $x \in \domR f$ and a direction of the form $+ \chi_i -\chi_j$, we say that 
$+ \chi_i -\chi_j$ is a \textit{steepest descent direction} of $f$ at $x$
if it minimizes the value $f'_\R(x;i, j)$ among all such directions.
 Denote 
by $\phi_\R(x)$  the slope of a steepest descent direction at $x$,
i.e., $\phi_\R(x) = \min_{i, j \in N}f'_\R(x;i, j)$.
 By the definition of $\phi_\R(x)$ and Theorem~\ref{thm:pM-minimizer}, 
we have $\phi_\R(x) \le 0$ for every $x \in \domR f$, 
and the equality holds if and only if $x$ is a minimizer of~$f$.

 By Theorem \ref{thm:pM-minimizer}, 
a minimizer of a polyhedral M-convex function can be found by
the steepest descent algorithm {\sc PM-LSD}, which is described in the same way as
{\sc M-LSD} for M-convex functions on $\Z^n$,
except that $x$ is a real vector (not necessarily integral), 
$f'(x;i,j)$ is replaced with $f'_\R(x;i,j)$, and
$\bar{c}(x;i,j)$ is replaced with $\bar{c}_\R(x;i,j)$ given by
\[
\bar{c}_\R(x;i, j) = \max\{\lambda \in \R_+ \mid 
 f(x + \lambda(\chi_i - \chi_j))- f(x) = \lambda f'_\R(x;i,j) \};
\]
 the value $\bar{c}_\R(x;i, j)$ is well defined since $f$ is a polyhedral convex function.
 It is not known so far whether the algorithm {\sc PM-LSD} terminates in
a finite number of iterations.

 We can show, as in  Section \ref{sec:Mconv}, that
$\phi_\R(x)$ is monotone non-decreasing in the algorithm {\sc PM-LSD}.
Proof is given in Section \ref{sec:proof:prop:pM-mj-sdd-mono} in Appendix.

\begin{prop}
\label{prop:pM-mj-sdd-mono}
 Let $y \in \dom f$ be a vector with $\phi_\R(y) < 0$,
$i,j \in N$ be distinct elements such that
$f'_\R(y;i,j) = \phi_\R(y)$, and
 $\lambda > 0$ be a real number
such that $f(y + \lambda(\chi_i - \chi_j)) - f(y) = \lambda \phi_\R(y)$.\\
{\rm (i)}
 The vector $\hat{y} = y+ \lambda(\chi_i - \chi_j)$
satisfies   $\phi_\R(\hat{y}) \ge \phi_\R(y)$.
\\
{\rm (ii)}
 For distinct $h,k \in N$, 
it holds that $f'_\R(\hat{y}; h,k) \ge \phi_\R(y)$.
 Moreover, if the inequality holds with equality, then
$+\chi_h - \chi_k$ is a steepest descent direction at $\hat{y}$
and satisfies $k \ne i$ and $h \ne j$.
\end{prop}

 To derive a finite bound on the number of iterations, 
we use algorithm {\sc PM-LSD2} and procedure {\sc PM-IncSlope}($x$), 
which are obtained by slight modification of
{\sc M-LSD2} and {\sc M-IncSlope}($x$)  for M-convex functions on $\Z^n$
as in the algorithm {\sc PM-LSD}.
 The following monotonicity property of
the procedure {\sc PM-IncSlope}$(x)$ can be obtained.

\begin{theorem}
\label{thm:pM-inc_min_slope}
 For a vector $x \in \dom f$ with $\phi_\R(x) < 0$,
the output $x'$ of the procedure {\sc PM-IncSlope}$(x)$ satisfies
$\phi_\R(x')  > \phi_\R(x)$.
\end{theorem}
\noindent
 Proof is given in Section \ref{sec:proof:thm:pM-inc_min_slope} 
 in Appendix.
 While the proof outline of Theorem \ref{thm:pM-inc_min_slope} is the same as
that of Theorem \ref{thm:M-inc_min_slope} for M-convex functions on $\Z^n$,
more careful analysis is required in the proof
due to the difference between domains $\Z^n$ and $\R^n$ of functions.

Since $f$ is a polyhedral convex function,
the directional derivative $f'_\R(x;i,j)$ can take only a finite number of values,
from which follows that $\{\phi_\R(x) \mid x \in \dom f\}$ is a finite set.
 This observation and Theorem \ref{thm:pM-inc_min_slope} imply
the finite termination of the algorithm {\sc PM-LSD2}.

\begin{theorem}
\label{thm:pM-runtime}
 For a polyhedral M-convex function $f: \R^n \to \Rinf$
with bounded $\domR f$,
the algorithm {\sc PM-LSD2} outputs a minimizer of $f$
in a finite number of iterations.
\end{theorem}

\appendix

\section*{Appendix}

\section{Omitted Proofs}

\subsection{Proof of Proposition \ref{prop:M-mj-sdd-mono}}
\label{proof:prop:M-mj-sdd-mono}

  Denote $\hat{y} =y+ \chi_i - \chi_j$.
  We prove the inequality $\phi(\hat{y}) \ge \phi({y})$ only
since the other claims in Proposition \ref{prop:M-mj-sdd-mono} can be obtained easily
from this inequality.
   Let $h,k \in N$ be distinct elements with $f'(\hat{y}; h,k) = \phi(\hat{y})$,
and denote 
$\widetilde{y} = \hat{y} +\chi_h -\chi_{k}$.
 We note that
\begin{align}
  f(\widetilde{y}) - f(y) = (f(\hat{y}) - f(y)) + (f(\widetilde{y}) - f(\hat{y})) 
=   \phi(y) +\phi(\hat{y}).
\label{eq:prop:M-mj-sdd-mono:4}
\end{align}

 Suppose first that $i \ne k$ and $j \ne h$.
 Since $\suppm(\widetilde{y}-y) = \{j,k\}$, 
 the condition (M-EXC) applied to $\widetilde{y}$, $y$, and $i \in \suppp(\widetilde{y}-y)$
implies that
\begin{align*}
 f(\widetilde{y}) + f(y) 
& \ge 
\min\{f(y + \chi_h - \chi_{k}) + f(y + \chi_i - \chi_{j}),
f(y + \chi_h - \chi_{j}) + f(y + \chi_i - \chi_{k})
 \} 
\\
& \ge 2(\phi(y) + f(y)),
\end{align*}
where the last inequality is by the definition of $\phi(y)$.
 This inequality, combined with \eqref{eq:prop:M-mj-sdd-mono:4},
implies $\phi(\hat{y}) \ge \phi(y)$.

 We then assume $i \ne k$ and $j = h$, implying that 
$\widetilde{y} = y+ \chi_i - \chi_k$.
 Since $+\chi_i  - \chi_j$ is a steepest descent direction at $y$,
we have $f(\widetilde{y})-f(y) \ge f(\hat{y}) - f(y) = \phi(y)$,
which, combined with \eqref{eq:prop:M-mj-sdd-mono:4},
implies $\phi(\hat{y}) = f(\widetilde{y})- f(\hat{y}) \ge 0 > \phi(y)$.
 The proof for the case with $i = k$ and $j \ne h$ is similar and omitted.

  We finally show that the case with $i = k$ and $j = h$ is not possible. 
 If $i = k$ and $j = h$, then $\widetilde{y} = y$ and therefore
\[
0 \ge\phi(\hat{y}) = f(\widetilde{y}) - f(\hat{y}) 
= f(y) - f(\hat{y})  = -f'(y;i,j) = - \phi(y) > 0,
\]
a contradiction.

\subsection{Proof of Proposition \ref{prop:constM-sdd-mono}}
\label{sec:constM-proofs}

Proposition \ref{prop:constM-sdd-mono} can be obtained by using the fact that 
the optimal value function $z(k)$ is
a convex function in $k$, i.e., the slope of $z(k)$ is monotone non-decreasing.

\begin{prop}
\label{prop:constM-z-convex}
It holds that
$z(k) - z(k-1) \le z(k+1) - z(k)$ 
 $(\underline{k} < k < \overline{k})$.
\end{prop}

\begin{proof}
 We prove the inequality $z(k-1) + z(k+1) \ge 2 z(k)$.
 Let $x \in M(k-1)$ and $\hat{x} \in M(k+1)$ be vectors that minimize
the value $\|x-\hat{x}\|_1$.
 Note that $f(x) = z(k-1)$ and $f(\hat{x}) = z(k+1)$.
 It suffices to show that there exists a vector $y \in \dom f$ with
$y(R)=k$ such that $f(x)+f(\hat{x}) \ge 2f(y)$ since $f(y) \ge z(k)$.

Since $\hat{x}(R) = k+1 > k-1 = x(R)$, we have
$\suppp(\hat{x}-x) \cap R \ne \emptyset$.
 By (M-EXC) applied to $\hat{x}, x$, and an arbitrarily chosen $i \in \suppp(\hat{x}-x) \cap R$,
there exists some $j \in \suppm(\hat{x}-x)$ such that
\begin{align}
\label{eqn:prop:constM-z-convex:1}
  f(\hat{x})+f(x) 
& \ge f(\hat{x}+ \chi_i -\chi_j)+f(x-\chi_i +\chi_j).
\end{align}

 Suppose that $j \in N \setminus R$.
 Then, we have
\begin{align*}
(\hat{x}+ \chi_i -\chi_j)(R) = (x-\chi_i +\chi_j)(R) = k.
\end{align*}
 By \eqref{eqn:prop:constM-z-convex:1}, it holds that
\begin{align*}
2\min\{f(\hat{x}+ \chi_i -\chi_j),f(x-\chi_i +\chi_j)\}  
\le f(\hat{x})+f(x). 
\end{align*}
 This shows that either of 
$y=\hat{x}+ \chi_i -\chi_j$ and $y=x-\chi_i +\chi_j$
satisfies the desired condition.

 To conclude the proof, we show that $j \in R$ is not possible.
 Assume, to the contrary, that $j \in R$.
 Then, it holds that
\[
(\hat{x}+ \chi_i -\chi_j)(R) = \hat{x}(R) = k+1, \qquad
(x-\chi_i +\chi_j)(R) = x(R) = k-1,
\]
from which follows that 
$f(\hat{x}+ \chi_i -\chi_j) \ge z(k+1)$ and $f(x-\chi_i +\chi_j) \ge z(k-1)$.
 These inequalities and \eqref{eqn:prop:constM-z-convex:1} imply that
\begin{align*}
z(k+1) + z(k-1) 
& = f(\hat{x})+f(x) \\
& \ge f(\hat{x}+ \chi_i -\chi_j)+f(x-\chi_i +\chi_j)
  \ge z(k+1) + z(k-1). 
\end{align*}
 Hence, all inequalities must hold with equality, i.e.,
we have $\hat{x}+ \chi_i -\chi_j \in M(k+1)$
and $x-\chi_i +\chi_j \in M(k-1)$.
 This, however, is a contradiction to the choice of $x$ and $\hat{x}$
since 
$\|(x-\chi_i +\chi_j) - \hat{x} \|_1 =\|x-\hat{x}\|_1 - 2$. 
\end{proof}

\begin{proof}[Proof of Proposition \ref{prop:constM-sdd-mono}]
 We have $x + \chi_i - \chi_j \in M(k+1)$ and $f'(x;i,j) = z(k+1) - z(k)$
by Proposition \ref{prop:constM-sdd-opt}.
 Proposition~\ref{prop:constM-z-convex} implies that
for every $h \in R$ and $\ell \in N\setminus R$, it holds that
\begin{align*}
f'(x + \chi_i - \chi_j; h,\ell) 
& \ge 
z(k+2) - z(k+1) \\
&
\ge z(k+1) - z(k)  = f'(x;i,j) = f'(x + \chi_i - \chi_j;i,j).   
\end{align*}
 This inequality, together with Proposition \ref{prop:constM-sdd-opt}, implies that
$x + 2\chi_i - 2\chi_j \in M(k+2)$.
\end{proof}

\subsection{Proof of Theorem \ref{thm:constM-inc_min_slope}}
\label{sec:proofs-constM}

We show that  $\phi^R(x')  > \phi^R(x)$ holds if $x'(R) < k$.
 The proof given below is similar to the one for Theorem \ref{thm:M-inc_min_slope}
for unconstrained M-convex function minimization.
  Let us consider Step 1 in some outer iteration
of the algorithm {\sc ConstM-IncSlope}$(x)$, and
let $i\in R$ be the element taken at the beginning of Step 1. 
The vector $y$ at the end of Step 1-1 satisfies the inequality $f'(y;i,j) > \phi^R(x)$
if ${y}(R) < k$.
 We first show that this inequality is preserved until the end
of the inner iterations in Step 1.

\begin{lemma}
\label{lem:constM-slope_inc_1}
Let $y \in \dom f$ be vectors with $\phi^R(y)=\phi^R(x)$,
and $i,j \in N$ be distinct elements such that $f'(y; i, j) > \phi(x)$.
 For $k \in N \setminus (R \cup \{j\})$ with $y+ \chi_i - \chi_k \in \dom f$, we have
$f'(y+ \chi_i - \chi_k;  i, j) > \phi^R(x)$.
\end{lemma}
 
\begin{proof}
 Let $\tilde{y} =y + \chi_i - \chi_k + \chi_i - \chi_j$.
 It suffices to show that $f(\tilde{y}) - f(y+ \chi_i - \chi_k) \ge f'(y; i, j)$
since $f'(y; i, j) > \phi(x)$.
 If $f(\tilde{y}) = + \infty$ then we are done; hence we assume $\tilde{y} \in \dom f$.
 By (M-EXC) applied to $y$, $\tilde{y}$, and $j \in \suppp(y- \tilde{y})$, it holds that
\begin{align*}
f(y) + f(\tilde{y}) 
& \ge  f(y - \chi_j + \chi_i) + f(\tilde{y} + \chi_j - \chi_i) 
=  f(y - \chi_j + \chi_i) + f(y + \chi_i - \chi_k)
\end{align*}
since $\suppm(y- \tilde{y})=\{i\}$.
 It follows that
\begin{align*}
f(\tilde{y}) - f(y+ \chi_i - \chi_k)
& \ge f(y + \chi_i - \chi_j) - f(y) = f'(y;i,j).
\end{align*}
\end{proof}

 Repeated application of Lemma~\ref{lem:constM-slope_inc_1} implies that 
vector $y$ at the end of Step 1 satisfies the inequalities
\begin{align}
\label{eqn:constM_slope_inc_1:1}
f'(y; i, j) > \phi^R(x) \qquad (j \in N\setminus R).
\end{align}

 Suppose that the inequalities \eqref{eqn:constM_slope_inc_1:1} for some $i \in R$
is satisfied by the vector $y$ at the end of Step 1 in some outer iteration.
 We then show that these  inequalities are preserved
in the following outer iterations.

\begin{lemma}
\label{lem:constM-slope_inc_2}
Let $y \in \dom f$ be vectors with $\phi^R(y)=\phi^R(x)$,
and $i \in R$ be an element satisfying $f'(y; i, j) > \phi^R(x)$ for every $j \in N\setminus R$.
Also, let $h \in R$ and $k \in N\setminus R$ be elements such that
$f'(y;h,k) = \phi^R(x)$.
 Then, 
$f'(y+\chi_h - \chi_k; i, j) > \phi^R(x)$ holds for every $j \in N\setminus R$.
\end{lemma}

\begin{proof}
 We fix $j^* \in N\setminus R$ and 
denote $\tilde{y}=y+\chi_h - \chi_k + \chi_i - \chi_{j^*}$. 
 It holds that  $i\ne k$ and $j^*\ne h$
since $i,h \in R$  and $k,j^* \in N\setminus R$.
 It suffices to show that 
\begin{align}
  \label{eqn:lem:constM-slope_inc_2-1}
f(\tilde{y}) - f(y+\chi_h - \chi_k) > \phi^R(x).
\end{align}
 If $f(\tilde{y}) = + \infty$ then we are done; hence we assume $\tilde{y} \in \dom f$.
  The condition (M-EXC) applied to $y$, $\tilde{y}$, and $k \in \suppp(y- \tilde{y})$ implies that 
\begin{align*}
f(y) + f(\tilde{y}) 
& \ge \min\{f(y - \chi_k  + \chi_h) + f(y+ \chi_i - \chi_{j^*}), 
f(y - \chi_k  + \chi_{i}) + f(y+ \chi_h - \chi_{j^*})\}\\
& \ge f(y - \chi_k  + \chi_h) + \min\{f(y+ \chi_i - \chi_{j^*}), f(y- \chi_k + \chi_i)\}\\
& > f(y + \chi_h  - \chi_k) + f(y)+\phi^R(x),
\end{align*}
where the second inequality is by the assumption 
$f'(y;h,k) = \phi^R(x)=\phi^R(y)$,
and the last inequality is by 
$f'(y; i, j) > \phi^R(x)$ $(j \in N\setminus R)$.
 Hence, \eqref{eqn:lem:constM-slope_inc_2-1} follows.
\end{proof}

 By repeated application of Lemma \ref{lem:constM-slope_inc_2}, we obtain
the inequalities $f'(x'; i, j) > \phi^R(x)$ 
$(i \in R,\ j \in N\setminus R)$ 
for the vector $x'$ at the end of the algorithm {\sc ConstM-IncSlope}$(x)$,
provided that $x'(R) < k$.
  Hence, the desired inequality $\phi^R(x') > \phi^R(x)$ follows.

\subsection{Proof of Proposition \ref{prop:pM-mj-sdd-mono}}
\label{sec:proof:prop:pM-mj-sdd-mono}

 To prove Proposition \ref{prop:pM-mj-sdd-mono}, we use the following property of
polyhedral M-convex functions, stating that the value of a function $f$ can be bounded from below
by using a local information at a given vector $x \in \dom f$.

\begin{prop}
\label{prop:pM_lbound_sdd}
  For $x, y \in \dom f$, it holds that
$f(y)-f(x)\ge (1/2) \|y-x\|_1 \, \phi_\R(x)$.
\end{prop}

\begin{proof}
 Since $f$ is polyhedral M-convex, 
there exist real numbers $\lambda_{ij} \ge 0\ (i,j \in N,\ i\ne j)$ 
such that
\begin{align*}
& \sum_{i,j \in N, i\ne j} \lambda_{ij}(\chi_i - \chi_j) = y-x,
\quad 
\sum_{i,j \in N,i\ne j} \lambda_{ij} = (1/2)\|y-x\|_1,
\\
& f(y) - f(x)
\ge 
\sum_{i,j \in N, i\ne j} \lambda_{ij} f'_\R(x; i,j)
\end{align*}
 (cf.~\cite[Theorem 4.15]{polyML}).
 We have $f'_\R(x; i,j) \ge \phi_\R(x)$ for distinct $i,j \in N$ by the definition of
 $\phi_\R(x)$. 
 Hence, the desired inequality $f(y)-f(x)\ge (1/2) \|y-x\|_1 \, \phi_\R(x)$
follows.
\end{proof}

  We first prove the inequality $\phi_\R(\hat{y}) \ge \phi_\R(y)$
in the statement (i), where $\hat{y} = y+ \lambda(\chi_i - \chi_j)$.
   Let $h, k \in N$ be distinct elements such that
$f'_\R(\hat{y}; h,k) = \phi_\R(\hat{y})$, and
${\mu} > 0$ be a real number with $\mu \le \lambda$ such that
\[
f(\hat{y} + \mu (\chi_h -\chi_{k})) - f(\hat{y}) = \mu \phi_\R(\hat{y}).
\]
 We denote $\tilde{y} = \hat{y} + {\mu} (\chi_h -\chi_{k})$.
 Then, we have
\begin{align}
  f(\tilde{y}) - f(y) = (f(\hat{y}) - f(y)) + (f(\tilde{y}) - f(\hat{y})) 
=  \lambda \phi_\R(y) + {\mu}\phi_\R(\hat{y}).
\label{eq:prop:pM-mj-sdd-mono:4}
\end{align}
 Since $(1/2)\|\tilde{y} - y \|_1 \le \lambda + \mu$ and $\phi_\R(y) < 0$,
Proposition \ref{prop:pM_lbound_sdd}
implies that
\[
f(\tilde{y}) - f(y) \ge (1/2) \|\tilde{y} - y \|_1 \phi_\R(y)
\ge (\lambda + \mu)\phi_\R(y).
\]
 It follows from this inequality and \eqref{eq:prop:pM-mj-sdd-mono:4} that
$\phi_\R(\hat{y}) \ge \phi_\R(y)$.

 We then prove the statement (ii).
  For every distinct $h,k \in N$, 
it holds that $f'_\R(\hat{y}; h,k) \ge \phi_\R(\hat{y}) \ge \phi_\R(y)$
by (i).
 If $f'_\R(\hat{y}; h,k) = \phi_\R(y)$ holds, then 
the two inequalities in $f'_\R(\hat{y}; h,k) \ge \phi_\R(\hat{y}) \ge \phi_\R(y)$ hold with equality,
and therefore $+\chi_h - \chi_k$ is a steepest descent direction at $\hat{y}$.
 Since $\phi_\R(\hat{y}) = \phi_\R(y)$,
the proof of (i) given above shows that 
$(1/2)\|\tilde{y} - y \|_1 \le \lambda + \mu$ holds with equality, from which
$k \ne i$ and $h \ne i$ follow.

\subsection{Proof of Theorem \ref{thm:pM-inc_min_slope}}
\label{sec:proof:thm:pM-inc_min_slope}

 We prove the inequality $\phi_\R(x')  > \phi_\R(x)$. 
 The proof outline is the  same as that for Theorem \ref{thm:M-inc_min_slope}.
  Hence, it suffices to show the following two lemmas that correspond to
Lemmas \ref{lem:M-slope_inc_1-2} and \ref{lem:M-slope_inc_2}.

\begin{lemma}
\label{lem:pM-slope_inc_1-2}
Let $y \in \dom f$ be vectors with $\phi_\R(y)=\phi_\R(x)$,
and $i,j \in N$ be distinct elements such that $f'_\R(y; i, j) > \phi_\R(x)$.
 For $k \in N \setminus\{i,j\}$ and $\lambda > 0$ with $\hat{y} \equiv y+ \lambda(\chi_i - \chi_k) \in \dom f$,
we have
$f'_\R(\hat{y};  i, j) \ge f'_\R(y; i, j)> \phi_\R(x)$.
\end{lemma}
 
\begin{proof}
 It suffices to show the inequality $f'_\R(\hat{y};  i, j) \ge f'_\R(y; i, j)$
since $f'_\R(y; i, j) > \phi_\R(x)$.
 Let $\tilde{y} =\hat{y} + \delta(\chi_i - \chi_j)$ with a sufficiently
small $\delta > 0$ such that
$f(\tilde{y}) - f(\hat{y}) = \delta f'_\R(\hat{y};i,j)$.
 By the choice of $\delta$, we have
\begin{align}
\label{eqn:lem:pM-slope_inc_1:1}
f(\hat{y} + \mu(\chi_i - \chi_j)) - f(\hat{y} + \mu'(\chi_i - \chi_j)) = (\mu-\mu') f'_\R(\hat{y};i,j)
\qquad (0 \le \mu' \le \mu \le \delta).
\end{align}

 We have $\suppp(y - \tilde{y}) = \{j,k\}$
and $\suppm(y - \tilde{y}) = \{i\}$.
 Hence, (M-EXC$[\R]$) applied to $y$, $\tilde{y}$, and $j \in \suppp(y- \tilde{y})$ implies that
there exists a sufficiently small $\epsilon > 0$ with $\epsilon \le \delta$ such that
\begin{align}
\label{eqn:lem:pM-slope_inc_1:2}
f(y) + f(\tilde{y}) 
& \ge  f(y - \epsilon(\chi_j - \chi_i)) + f(\tilde{y} + \epsilon(\chi_j - \chi_i)).
\end{align}
 Since $\epsilon$ is sufficiently small, we have
\begin{align}
\label{eqn:lem:pM-slope_inc_1:3}
f(y - \epsilon(\chi_j - \chi_i)) - f(y)  
=f(y + \epsilon(\chi_i - \chi_j)) - f(y)  
= \epsilon f'_\R({y};i,j).
 \end{align}
 By \eqref{eqn:lem:pM-slope_inc_1:1}, \eqref{eqn:lem:pM-slope_inc_1:2},
and \eqref{eqn:lem:pM-slope_inc_1:3}, we have
\begin{align*}
\epsilon f'_\R(\hat{y};i,j)
& = f(\hat{y} + \delta(\chi_i - \chi_j)) - f(\hat{y} + (\delta-\epsilon)(\chi_i - \chi_j))\\
& = f(\tilde{y}) - f(\tilde{y} + \epsilon(\chi_j - \chi_i))\\
& \ge f(y - \epsilon(\chi_j - \chi_i)) - f(y) =\epsilon f'_\R({y};i,j).
\end{align*}
Hence, $f'_\R(\hat{y};  i, j) \ge f'_\R(y; i, j)$ follows.
\end{proof}

\begin{lemma}
\label{lem:pM-slope_inc_2}
Let $y \in \dom f$ be vectors with $\phi_\R(y)=\phi_\R(x)$,
and $i \in N$ be an element satisfying
$f'_\R(y; i, j) > \phi_\R(x)$ for every $j \in N\setminus\{i\}$.
Also, let $h, k \in N$ be distinct elements such that
$f'_\R(y;h,k) = \phi_\R(x)$.
 Then, for every $\lambda \in \R$ with $0< \lambda \le \bar{c}_\R(x;i,j)$, 
 the vector $\hat{y}\equiv y+\lambda(\chi_h - \chi_k)$ satisfies 
$f'_\R(\hat{y}; i, j) > \phi_\R(x)$ for every $j \in N\setminus\{i\}$.
\end{lemma}

\begin{proof}
 We fix $j^* \in N\setminus\{i\}$ and prove the inequality $f'_\R(\hat{y};  i, j^*) > \phi_\R(x)$.
 We may assume that $f'_\R(\hat{y};i,{j^*}) = \phi_\R(\hat{y})$
since otherwise $f'_\R(\hat{y};i,{j^*}) > \phi_\R(\hat{y}) \ge \phi_\R(y) = \phi_\R(x)$
by Proposition \ref{prop:pM-mj-sdd-mono} (i).
 This assumption implies that $i\ne k$ and ${j^*}\ne h$
by Proposition \ref{prop:pM-mj-sdd-mono} (ii).

 Since $0< \lambda \le \bar{c}_\R(x;i,j)$, we have
\begin{align}
f(\hat{y}) - f({y})
= \lambda f'_\R(y;h,k) = \lambda \phi_\R(x).
\label{eqn:lem:pM-slope_inc_2:1}
\end{align}
 Let $\tilde{y} =\hat{y} + \delta(\chi_i - \chi_{j^*})$ with a sufficiently
small $\delta > 0$ such that
$f(\tilde{y}) - f(\hat{y}) = \delta f'_\R(\hat{y};i,{j^*})$.


  Since $\suppp(\tilde{y}-y)=\{i,h\}$ and $\suppm(\tilde{y}-y)=\{{j^*},k\}$,
 (M-EXC$[\R]$) applied to $\tilde{y}$, $y$, and $i\in \suppp(\tilde{y}-y)$ implies that 
there exists a sufficiently small $\epsilon > 0$ with $\epsilon < \min(\lambda, \delta)$ such that
\begin{align}
& f(\tilde{y}) + f(y)  
\notag\\
& \ge \min\{
f(\tilde{y}- \epsilon(\chi_i  - \chi_{j^*})) + f(y+\epsilon(\chi_i  - \chi_{j^*})),  
f(\tilde{y}- \epsilon(\chi_i  - \chi_k)) + f(y+\epsilon(\chi_i  - \chi_k))  
\}
\notag\\
& = \min\{
f(\tilde{y}- \epsilon(\chi_i  - \chi_{j^*})) + \epsilon f'_\R(y;i,{j^*}),  
f(\tilde{y}- \epsilon(\chi_i  - \chi_k)) + \epsilon f'_\R(y;i,k)  
\}  + f(y)  
\notag\\
& > 
\min\{
f(\tilde{y}- \epsilon(\chi_i  - \chi_{j^*})), 
f(\tilde{y}- \epsilon(\chi_i  - \chi_k)) 
\}
 + f(y)  + \epsilon \phi_\R(x),
\label{eqn:lem:pM-slope_inc_2:4}
\end{align}
 where the equality holds since $\epsilon$ is a sufficiently small positive number,
and the strict inequality is by $f'_\R(y; i, j) > \phi_\R(x)$  $(j \in N\setminus\{i\})$.
  By Proposition \ref{prop:pM_lbound_sdd}
and the equation 
\[
\|(\tilde{y}- \epsilon(\chi_i  - \chi_{j^*})) - {y}\|_1
= 
\|(\tilde{y}- \epsilon(\chi_i  - \chi_{k})) - {y}\|_1
= 2(\lambda + \delta - \epsilon),
\]
we have
\begin{align*}
 \min\{
f(\tilde{y}- \epsilon(\chi_i  - \chi_{j^*})), 
f(\tilde{y}- \epsilon(\chi_i  - \chi_k)) 
\} - f(y)
&  \ge
(\lambda + \delta - \epsilon) \phi_\R(y) 
 = (\lambda + \delta - \epsilon) \phi_\R(x),
\end{align*}
which, combined with \eqref{eqn:lem:pM-slope_inc_2:4}, 
implies 
$f(\tilde{y}) -f(y) >  (\lambda + \delta) \phi_\R(x)$.
  It follows from this inequality and \eqref{eqn:lem:pM-slope_inc_2:1} that
\begin{align*}
\delta f'_\R(\hat{y};i,{j^*})
 & =  f(\tilde{y}) - f(\hat{y})
  > (\lambda + \delta) \phi_\R(x) - \lambda \phi_\R(x) = 
 \delta \phi_\R(x).
\end{align*}
 Hence, the inequality $f'_\R(\hat{y};i,{j^*}) > \phi_\R(x)$ follows.
\end{proof}

\section*{Acknowledgment}
This work was partially supported by JST ERATO Grant Number JPMJER2301,
JST FOREST Grant Number JPMJFR232L, and
JSPS KAKENHI Grant Numbers JP22K17853, 23K10995, and JP24K21315.

\end{document}